\documentclass[leqno,12pt]{article}

\usepackage{amsmath, amsthm, amsfonts, amssymb, latexsym, amscd, stackrel,url,hyperref}
\usepackage{enumerate}

\usepackage{xy}
\xyoption{all}

\usepackage[mmddyyyy]{datetime}

\makeatletter
\def\url@smallstyle{%
  \@ifundefined{selectfont}{\def\UrlFont{\sf}}{\def\UrlFont{\small\ttfamily}}}
\makeatother
\numberwithin{equation}{section}

\theoremstyle{plain}
    \newtheorem{theorem}[equation]{Theorem}
    \newtheorem{lemma}[equation]{Lemma}
    \newtheorem{corollary}[equation]{Corollary}
    \newtheorem{proposition}[equation]{Proposition}

    \newtheorem*{theorem*}{Theorem}
    \newtheorem*{proposition*}{Proposition}
    \newtheorem*{corollary*}{Corollary}
    \newtheorem*{lemma*}{Lemma}
    \newtheorem*{conjecture*}{Conjecture}
    \newtheorem{definition-theorem}[equation]{Definition/Theorem}
    \newtheorem{definition-lemma}[equation]{Definition/Lemma}
\theoremstyle{definition}
    \newtheorem{definition}[equation]{Definition}
    \newtheorem{example}[equation]{Example}

    \newtheorem{remark}[equation]{Remark}
    
    \newtheorem{remarks}[equation]{Remarks}

    \newcommand{\R}{\mathbb{R}}
    \newcommand{\C}{\mathbb{C}}
    \newcommand{\N}{\mathbb{N}}

    \renewcommand{\H}{{\mathcal H}}
   	\renewcommand{\phi}{\varphi}
    \let\epsilon\varepsilon

    \newcommand{\Bounded}{\mathfrak{B}}
    \newcommand{\Compact}{\mathfrak{K}}

\newcommand{\argument}{\underbar{\phantom{a}}}

\DeclareMathOperator{\Hom}{Hom}

\newcommand{\SL}{\operatorname{SL}}

\DeclareMathOperator{\id}{id}

\DeclareMathOperator{\lspan}{span}
\DeclareMathOperator{\image}{image}

\DeclareMathOperator{\Mod}{Mod}

\DeclareMathOperator{\Ind}{Ind}
\DeclareMathOperator{\Res}{Res}

\newcommand{\alg}{\mathrm{alg}}
\newcommand{\cb}{\mathrm{cb}}

\newcommand{\coloneq}{:=}

\begin{document}
\title{\texorpdfstring{Adjoint functors between categories of Hilbert $C^*$-modules}{Adjoint functors between categories of Hilbert C*-modules}}
\author{Pierre Clare \and Tyrone Crisp \thanks{Partially supported by the Danish National Research Foundation through the Centre for Symmetry and Deformation (DNRF92).} \and Nigel Higson \thanks{Partially supported by the US National Science Foundation DMS-1101382. }}

\date{}
\def\parsedate #1:20#2#3#4#5#6#7#8\empty{#4#5/#6#7/20#2#3}
\def\moddate#1{\expandafter\parsedate\pdffilemoddate{#1}\empty}

\maketitle

\begin{abstract}
Let $E$ be a (right) Hilbert module over a $C^*$-algebra $A$.  If $E$ is equipped with a left action of a second $C^*$-algebra $B$, then tensor product with $E$ gives rise to a functor from the category of Hilbert $B$-modules to the category of Hilbert $A$-modules.  The purpose of this paper is to study adjunctions between functors of this sort.  We shall introduce a new kind of adjunction relation, called a local adjunction, that is weaker than the standard concept from category theory. We shall give several examples, the most important of which is the functor of parabolic induction in the tempered representation theory of real reductive groups.  Each  local adjunction gives rise to an ordinary adjunction of functors between categories of Hilbert space representations.  In this way we shall show that the parabolic induction functor has a simultaneous left and right adjoint, namely the parabolic restriction functor constructed in \cite{CCH_ups}.
\end{abstract} 



\section{Introduction}\label{intro_section}

Let $A$ be a $C^*$-algebra and denote by ${}_A\H$ the category of non-degenerate Hilbert space representations of $A$.   This is obviously a category of interest when $A$ is   a group $C^*$-algebra, since it is isomorphic to the category of unitary   representations of the group.   Similarly, if $A$ is the \emph{reduced} $C^*$-algebra of a real reductive group, then ${}_A\H$ is isomorphic to the category of \emph{tempered} unitary representations of the group.

In addition, denote by $\H_A$ the category of right Hilbert $A$-modules (a Hilbert module is a  Banach module whose norm is obtained from an associated $A$-valued inner product; see  \cite{BlM}, \cite{Lance}, or Section~\ref{background_section} for a quick review).   The role of this category in representation theory is a bit less clear, but for example if $A$ is the reduced $C^*$-algebra of a real reductive group $G$, then each discrete series representation of $G$ determines an object in $\H_A$.  More generally, the category $\H_A$ captures the topology of the dual space $\widehat A$ of irreducible representations  of $A$ (this is the tempered dual in our reduced group $C^*$-algebra example).    In contrast, the category ${}_A\H$ of Hilbert space representations is more closely related to the structure of the dual as a set or measurable space.

In this paper we shall study the categories  ${}_A \H$ and  $\H_A$  with a particular view to the case of   the reduced $C^*$-algebra of a real reductive Lie group $G$. We shall further develop the  analysis of the parabolic induction functor 
\[
\Ind_P^G  \colon {}_{C^*_r (L)}\H  \longrightarrow {}_{C^*_r (G)}\H
\]
from the perspective of Hilbert modules that was  begun in \cite{Clare_pi} and \cite{CCH_ups}.  Here  $L$ is a Levi subgroup of a parabolic subgroup   $P\subseteq G$.  We shall examine the relationship between parabolic induction and the functor of parabolic restriction 
 \[
\Res^G_P  \colon {}_{C^*_r (G)}\H  \longrightarrow {}_{C^*_r (L)}\H
\]
that we introduced in \cite{CCH_ups}.  We shall prove that parabolic induction and restriction are left and right adjoints of one another. In fact we shall prove this as a consequence of a stronger statement involving the Hilbert module categories $\H_{C^*_r(G)}$ and $\H_{C^*_r (L)}$.  See Theorems~\ref{main-reductive-thm} and \ref{reductive-Hilbertspace-theorem}.

If $A$ and $B$ are $C^*$-algebras, then by a \emph{correspondence from $A$ to $B$} we shall mean a Hilbert $B$-module equipped with a nondegenerate representation of $A$ as adjointable Hilbert $B$-module endomorphisms.  Associated to such a correspondence $F$ there are functors 
\[
{}_B\H \longrightarrow {}_A\H\quad \text{and} \quad \H_A \longrightarrow \H_B
\]
that are constructed using Hilbert module tensor products.  The functors of parabolic induction and restriction from \cite{Clare_pi} and \cite{CCH_ups} are of this type. 

In the algebraic setting, where $A$ and $B$ are rings with unit, it is well known that if $F$ is an $A$-$B$-bimodule, then the associated   tensor product functors 
\[
{}_B\operatorname{Mod}  \longrightarrow {}_A\operatorname{Mod}
\quad \text{and} \quad \operatorname{Mod}_A \longrightarrow \operatorname{Mod}_B
\]
between left and right module categories  always admit  right adjoints, namely $X \mapsto \Hom_A(F,X) $ and $Y\mapsto \Hom_B(F,Y)$ respectively. But in the $C^*$-setting the extra symmetry imposed by the $*$-operation means that every right adjoint is also a left adjoint, and the existence of an  adjoint functor  in these circumstances is a much more delicate matter.

Kajiwara, Pinzari and Watatani \cite{KPW} have obtained necessary and sufficient conditions on a correspondence $F$ for the tensor-product functor between Hilbert module categories to admit an adjoint. Unfortunately their conditions, which we shall review in Section~\ref{full_adj_section},  immediately rule out many naturally-occurring examples of correspondences. Most notably from our point of view, the conditions are not satisfied by the correspondences associated to parabolic induction and restriction. Comparing with the theory of smooth representations of $p$-adic reductive groups---where, as Bernstein showed  (\cite{Bernstein}, cf. \cite{Renard}),  the parabolic-induction functor admits both left and right adjoints, and where those adjoints are close to being identical---we were led to look for a weaker notion of adjunction for Hilbert modules that  would apply in particular to parabolic induction. The main novelty  of this paper is to define such a notion, and to study some of its properties. 

If $F$ is a correspondence from $A$ to $B$, and $E$ is a correspondence from $B$ to $A$, we shall say that the associated tensor-product functors 
\[
 \H_A\stackrel{\otimes _A F}\longrightarrow  \H_B \quad \text{and} \quad   \H_B\stackrel{\otimes _B E}\longrightarrow \H_A
 \]
  are \emph{locally adjoint} if  there are natural isomorphisms
\[ 
\Compact_B(  X\otimes _A F , Y)\stackrel \cong \longrightarrow  \Compact_A(X,  Y\otimes _B E )
\]
between   spaces of  \emph{compact} Hilbert module  operators.     See Definition \ref{ladj_definition} for   details. Not every morphism between Hilbert modules is compact, so our definition is not the usual definition of an adjunction.  Nor is it obtained from the usual definition by adjusting   to take into account the $C^*$-structure \cite{GLR} on the categories $\H_A$ and $\H_B$.  But if  two tensor product  functors  are adjoint in the usual sense, then the natural isomorphisms   coming from the adjunction restrict to  isomorphisms between spaces of  compact operators, as above, and so adjunction implies local adjunction. See Corollary~\ref{adj_implies_ladj_corollary}.  The converse is rarely true: in general local adjunction is a genuinely weaker condition. 

The commutative case illustrates the distinction between adjoint and local adjoint quite well.  If $Y$ is a quotient of a compact Hausdorff space $X$, then the pullback functor $\H_{C(Y)}\to \H_{C(X)}$ has an adjoint if and only if $X$ is a finite cover of $Y$, while it has  a local adjoint if and only if $X$ is a finite \emph{branched} cover. See Section~\ref{sec-commutative}.  In addition, if $U$ is an open subset of $X$, then the pushforward $\H_{C_0(U)}\to \H_{C(X)}$ always has a local adjoint, but it has an adjoint if and only if $U$ is both open and  closed in $X$.

In Sections~\ref{adj_bimod_section} and \ref{sec-indexes} we  characterize locally adjoint pairs of tensor product functors, and show that the condition of local adjointability is equivalent to the condition of ``finite numerical index'' considered in \cite{KPW}.   
Returning  to categories of Hilbert space representations, we show that if two tensor product functors  are locally adjoint, then the associated functors between the Hilbert space representation categories ${}_A\H$ and ${}_B\H$  are (two-sided) adjoints in the usual sense.  See Theorem~\ref{rep-adjunction-theorem}.
 
 On the basis of all  this it is easy to analyze several examples, and we shall do so in Section~\ref{examples_section}.  Our main example of parabolic induction is considered in Section~\ref{parabolic_section}.  As a particular consequence  we show that the parabolic restriction functor constructed in \cite{CCH_ups} is a two-sided adjoint 
to the parabolic induction functor from tempered unitary representations of a Levi factor of a reductive group $G$ to tempered unitary representations of $G$. This adjoint functor appears to be new.

\section{Background and Notation}\label{background_section}

\subsection{Hilbert Modules and Correspondences}

Let $A$ be a $C^*$-algebra. Recall that a   \emph{Hilbert $A$-module} is a right $A$-module $X$ that is equipped with an inner product 
\[
\langle\,\argument\,,\,\argument\,\rangle :X\times X\longrightarrow  A
\]
that is $\C$-antilinear in the first variable, $\C$-linear in the second variable, and satisfies
\[ \langle x,y a\rangle = \langle x,y\rangle  a,\qquad \langle x,y\rangle ^* = \langle y, x\rangle \quad \text{and}\quad  \langle x,x\rangle \geq 0\]
for all $x,y\in X$ and $a\in A$.  In addition the  formula 
\[
\|x\|_X^2 \coloneq \|\langle x,x\rangle  \|_A 
\]
is required to define  a complete norm on $X$.  See for example \cite[Chapter 8]{BlM} or \cite{Lance} for an introduction to this concept. (But note that the name ``Hilbert $C^*$-module'' is used in \cite{BlM} to refer to what we are calling Hilbert module, while the  name ``Hilbert module''  is used in \cite{BlM} to refer to a different concept.)

A map $T:X\to Y$ between Hilbert $A$-modules is \emph{adjointable} if there is a map $T^*\colon Y \to X $ (necessarily unique) satisfying 
\[
\langle Tx,y\rangle  = \langle x,T^*y\rangle 
\]
for all $x\in X$ and $y\in Y$.  Adjointable operators are automatically $A$-linear and bounded, but the converse is not true, in general. We shall denote by $\Bounded_A(X,Y)$ the Banach space of all adjointable operators from  $X$ to $Y$. The space $\Bounded_A(X,X)$ of adjointable operators from $X$ to itself is a $C^*$-algebra in the operator norm.  

\begin{definition}
\label{nondegen-def}
Let $B$ be a second $C^*$-algebra and let $F$ be a Hilbert $B$-module. A $*$-homomorphism $A\to \Bounded_B(F,F)$ is \emph{nondegenerate} if the elements $af$ ($a\in A$, $f\in F$) span a dense subspace of $F$. 
\end{definition}

\begin{definition}
By a \emph{correspondence from $A$ to $B$} we shall mean a Hilbert $B$-module $F$ equipped with  a nondegenerate $*$-homomorphism from $A$ into $\Bounded_B(F,F)$.  
\end{definition}

\begin{remarks} 
A correspondence from $A$ to $B$ is in particular an $A$-$B$ bimodule, but the definition is asymmetric in that  no  $A$-valued inner product is implied.  The condition of nondegeneracy will not play a crucial role in what follows.  Nevertheless we shall assume it since it in any case holds in the main examples of interest to us.
\end{remarks}

\subsection{Compact Operators}

If $X$ and $Y$ are Hilbert $A$-modules, and if $x\in X$ and $y\in Y$, then the formula
\begin{equation}
\label{eq-rank-one1}
y\otimes x^* \colon z \mapsto  y\langle x,z\rangle
\end{equation}
defines an adjointable operator from $X$ to $Y$; the adjoint is $x\otimes y^*$.
The operator-norm closure of the linear span of the operators 
\begin{equation}
\label{eq-rank-one2}
y\otimes x^* \in \Bounded_A(X,Y)
\end{equation}
is by definition the subspace of \emph{$A$-compact} operators, denoted $\Compact_A(X,Y)$. 
The composition, on either side, of a compact operator with an adjointable operator is compact.

To add some algebraic substance to the notation used in (\ref{eq-rank-one1}) and (\ref{eq-rank-one2}), we introduce the following concept:

 \begin{definition}
\label{conjugate_definition1}
Let $X$ be a {Hilbert $A$-module}.   The \emph{conjugate} of $X$, denoted $X^*$, is the complex conjugate of the vector space $X$, equipped with left $A$-module structure  defined by 
\[
a\cdot x^*   = (x a^* )^*.
\]
Here $x^*\in X^*$ denotes the image of $x\in X$ under the obvious conjugate $\C$-linear isomorphism $X\to X^*$.
 \end{definition}

\begin{remark}
The term \emph{adjoint} would usually be a more appropriate name for $X^*$ than \emph{conjugate}, but   the word will be used quite heavily in other senses in this paper. \end{remark}

Using the  left $A$-module structure on $X^*$ we find that 
\[
y\cdot a \otimes x^* = y \otimes a\cdot x^* \in \Bounded_A(X,Y) ,
\]
and so the formula (\ref{eq-rank-one1}) defines a $\C$-linear map 
\begin{equation}
\label{eq-tp-cpts}
 Y\otimes _A^{\text{alg}} X^* \longrightarrow \Bounded_A (X,Y) .
\end{equation}
The compact operators therefore constitute a completion of the algebraic tensor product  $Y\otimes _A^{\text{alg}} X^*$ within $\Bounded _A (X,Y)$.  We shall return to this perspective on compact operators in a little while.

 We shall need one more idea about compact operators.  Let $X$ be a Hilbert $A$-module.  The   space $\Compact _A (A,X)$ of compact adjointable operators from the Hilbert module $A$ to $X$ is itself a Hilbert $A$-module under the inner product 
 \[
 \langle S, T\rangle = S^* T
 .
 \]
 Here the operator $S^*T\colon A \to A$ is left multiplication by some unique element of $A$, and we identify $S^*T$ with that element so as to obtain an $A$-valued inner product.  Each operator in $\Compact _A (A,X)$ is of the form 
 \[
 T \colon a \mapsto xa
 \]
 for some unique element $x\in X$. Moreover all these operators are adjointable, with adjoint 
 \[
 T^* \colon y \mapsto \langle x ,y\rangle.
 \]
Finally, $T$ and $T^*$  are compact,   as can be seen using an approximate identity in $A$.  We arrive at the following results (see also \cite[Proposition 8.1.11]{BlM}). 
 
 \begin{lemma}
 \label{lem-compacts}
 The Hilbert $A$-module $X$ is isomorphic to the Hilbert $A$-module  $\Compact _A (A,X)$ via the map  that associates to $x\in X$ the compact operator $a \mapsto x a$. \qed
 \end{lemma}
 
 \begin{lemma}
 \label{lem-compacts2}
  The conjugate $X^*$ of a Hilbert $A$-module $X$ is isomorphic to the   $A$-module  $\Compact _A (X,A)$ via the map  that associates to $x^*\in X^*$ the compact operator $y \mapsto \langle x,y\rangle$. \qed
 \end{lemma}
 
\subsection{Functors on Hilbert Modules}

\begin{definition}
If $A$ is  a  $C^*$-algebra, then  denote by  $\H_A$   the category whose objects are right Hilbert $A$-modules and whose morphisms are adjointable maps between Hilbert modules.  
\end{definition}

Within the algebraic context, bimodules give rise to functors between module categories via tensor product.  This is so in the Hilbert module context, too, thanks to the following construction.

\begin{definition}
\label{def-itp}
Let $X$ be a Hilbert $A$-module and let $F$ be a correspondence from $A$ to $B$. The (\emph{internal}) \emph{tensor product} $X\otimes_A F$, which is a Hilbert $B$-module,  is the completion of  the algebraic tensor product $X\otimes^{\alg}_A F$ in the norm induced by the $B$-valued inner product
\[ 
\langle x_1\otimes f_1, x_2\otimes f_2\rangle_{X\otimes _A F}  \coloneq \langle f_1, \langle x_1,x_2\rangle_X \cdot f_2\rangle_F.
\]
\end{definition}

 See \cite[Chapter 4]{Lance} for basic information on the internal tensor product construction. One has natural isomorphisms 
 \[
 A\otimes_A F\stackrel{\cong}\longrightarrow  F\quad \text{and}\quad  F\otimes_B B  \stackrel{\cong}\longrightarrow  F
 \]
  via multiplication.

If $F$ is a correspondence from $A$ to $B$, then internal tensor product with $F$ gives rise to a \emph{tensor product functor} 
\[
F \colon \H_A \longrightarrow \H_B
\]
(as indicated, we shall use the same letter for the bimodule and the functor), since the tensor product of an adjointable operator with the identity operator on $F$ is an adjointable operator between tensor product modules. 

It is interesting to note that  subject to a natural continuity condition and compatibility with the adjoint operation, every functor between Hilbert module categories is a tensor product functor.  We shall not use this fact, but here is a short summary.

\begin{definition}
A functor $ F$ between  categories of Hilbert modules is called a \emph{$*$-functor} if it is $\C$-linear on morphisms, and satisfies $  F(T^*)=  F(T)^*$ for every adjointable operator $T$. 
\end{definition}

\begin{definition} A $*$-functor $F :\H_A\to\H_B$ is \emph{strongly continuous} if for every object $X\in \H_A$, the $*$-homomorphism
\[
  F: \Compact_A(X,X) \to \Bounded_B(F(X), F(X) )
  \]
   is nondegenerate. 
\end{definition}

\begin{theorem}\textup{\cite[Theorem 5.3]{Blecher}}
\label{Blecher_theorem}
The category of strongly continuous $*$-functors $\H_A\to \H_B$ \textup{(}and natural transformations\textup{)} is   equivalent to the category of correspondences from $A$ to $B$ \textup{(}and adjointable operators compatible with the $A$-$B$ bimodule structure\textup{)}.
\end{theorem}

The equivalence is given in one direction by sending a functor $F$ to the correspondence $ F(A)$, and in the other direction by sending a correspondence $F$ to the associated tensor product functor. 

\begin{remark}
\label{rem-blecher}
In addition  a strongly continuous $*$-functor $\H_A\to \H_B$ is an equivalence (with inverse a strongly continuous $*$-functor) if and only if the associated correspondence is a \emph{Morita equivalence} between $A$ and $B$: see \cite[8.1.2, 8.2.20]{BlM}. If $A$ and $B$ are Morita equivalent, then the dual spaces $\widehat{A}$ and $\widehat{B}$ are homeomorphic (moreover the  converse also holds if $A$ and $B$ are commutative).  So  we see that $\H_A$ carries information about the structure of $\widehat A$ as a topological space, as remarked in the introduction.
\end{remark}

\subsection{Adjunctions}
\label{full_adj_section}

We shall start with the standard definition from category theory.

\begin{definition}\label{adj_definition}
Let $A$ and $B$ be $C^*$-algebras, and let $E$ and $F$ be correspondences from $B$ to $A$ and from $A$ to $B$, respectively, determining tensor product functors  
 \[
 \H_A\stackrel{\otimes _A F}\longrightarrow  \H_B \quad \text{and} \quad   \H_B\stackrel{\otimes _B E}\longrightarrow \H_A .
 \] 
 An \emph{adjunction} between $ F$ and $ E$ is a natural isomorphism
\begin{equation}
\label{eq-adjunction}
 \Phi_{X,Y}:\Bounded_B( X\otimes _A F , Y)\stackrel \cong \longrightarrow  \Bounded_A(X,  Y\otimes _B E),
 \end{equation}
or in other words  a natural equivalence between the left and right sides in (\ref{eq-adjunction}), considered as functors from the product category $\H_A^{\mathrm{op}} \times \H_B$ to the category of sets. 
  \end{definition}

 
But unlike the ordinary situation in category theory, there is no real distinction between left adjoints  and right adjoints in the context of Hilbert modules: given an adjunction $\Phi_{X,Y}$ as above, the formula 
\[
\Phi_{X,Y}^*:T\mapsto (\Phi_{X,Y})^{-1}(T^*)^*
\]
defines  an adjunction
\begin{equation}
\label{eq-psi-from-phi}
\Phi^*_{X,Y} \colon 
\Bounded_A(  Y\otimes _B E, X) \stackrel \cong \longrightarrow  \Bounded_B( Y,X\otimes _A F  )
\end{equation}
that reverses the role of $E$ and $F$ in Definition~\ref{adj_definition}.

\begin{definition}
\label{def-unit-counit}
Let $\Phi$ be an adjunction, as in Definition~\ref{adj_definition}.
 \begin{enumerate}[\rm (a)]
   \item A \emph{unit} for $\Phi$ is a bounded, adjointable $A$-bimodule map 
  \[
 \eta\colon  A \longrightarrow F\otimes _B E
  \]
that defines $\Phi_{X,Y}$ by means of the commuting diagram 
\[
\xymatrix@C=40pt{
 X\otimes _A A \ar[d]_{\cong} \ar[r]^-{\mathrm{id}_X \otimes \eta}  &  X\otimes _A F \otimes _B E \ar[d]^{ T \otimes \mathrm{id}_E} \\
 X \ar[r]_-{\Phi_{X,Y}(T)}  & Y\otimes _B E .
 }
\]
 
  \item A \emph{counit} for  $\Phi$  is a bounded, adjointable $B$-bimodule map  
  \[
 \varepsilon \colon E\otimes _A F  \longrightarrow B
  \]
that defines the inverse of the isomorphism $ \Phi_{X,Y}$ in (\ref{eq-adjunction}) by means of the commuting diagram 
\[
\xymatrix@C=40pt{
 Y\otimes _B B\ar[d]_{\cong}   &  Y\otimes _B E \otimes _A F 
 \ar[l]_-{\mathrm{id}_Y \otimes \varepsilon}
\\
Y  & X\otimes _B F  \ar[l]^-{\Phi_{X,Y}^{-1}(T)}  \ar[u]_{ T \otimes \mathrm{id}_F} .
 }
\]

  \end{enumerate}
\end{definition}

\begin{proposition}\label{adj_unit_counit_proposition}
 Every adjunction admits a unique unit and counit.
\end{proposition}

\begin{proof}
This is standard.  For instance the unit is the image of the identity operator on $F$  under the map 
\[
 \Phi_{A,F}:\Bounded_B(   F , F)\stackrel \cong \longrightarrow  \Bounded_A(A,  F\otimes _B E),
\]
See \cite[IV.1]{MacLane}, for example. \end{proof}

No continuity conditions are imposed on the isomorphisms $\Phi_{X,Y}$ in an adjunction. But in fact continuity is automatic, as the following calculation shows.  

\begin{lemma}\label{adj_cb_lemma}
Let $\Phi$ be an adjunction, as in Definition~\ref{adj_definition}.  Each map 
\[
 \Phi_{X,Y}:\Bounded_B( X\otimes _A F , Y) \stackrel \cong \longrightarrow  \Bounded_A(X,  Y\otimes _B E),
\]
 is a linear, topological  isomorphism.
\end{lemma}

\begin{proof}[Proof of the Lemma]
It is clear from the definition of unit that $\Phi_{X,Y}$ is linear, with norm bounded by the operator norm of the unit map from $A$ into 
$F\otimes_B E$.  The inverse is likewise linear, with norm bounded by the norm of the counit.
\end{proof}

The following definition very slightly elaborates on Definition~\ref{conjugate_definition1}.

\begin{definition}
\label{conjugate_definition2}
Let $A$ and $B$ be $C^*$-algebras and let $F$ be a {correspondence from $A$ to $B$}.   The \emph{conjugate} of $F$ is the Hilbert module conjugate  $F^*$, as in Definition~\ref{conjugate_definition1}, equipped with right $A$-module structure defined by 
\[
 f^* \cdot a = (a^*f )^*.
\]
 \end{definition}

The conjugate $F^*$ of a correspondence from $A$ to $B$ is a $B$-$A$-bimodule, but, as it stands,  is not a correspondence from $B$ to $A$: there is no obvious $A$-valued inner product. Kajiwara, Pinzari and Watatani \cite{KPW} relate the existence of such an inner product on $F^*$   to the existence of adjoint functors, as follows.

\begin{theorem}\textup{\cite[Theorem 4.4(2), Theorem 4.13]{KPW}}
\label{adj_bimod_theorem}
 A  tensor product functor   
from $
   \H_A$ to $ \H_B
$,
 induced from a correspondence $F$,    has an adjoint tensor product functor if and only if all the following conditions are met:
 \begin{enumerate}[\rm (a)]
 
 \item The conjugate  $B$-$A$-bimodule  $ {F}^*$ carries an $A$-valued inner product making  it  into a correspondence from $B$ to $A$.
 
 \item  The conjugate operator space structure  on $F^*$ is completely boundedly equivalent to the Hilbert $A$-module operator space structure on $F^*$.
 
  \item The left action of $A$ on $F$ is through a $*$-homomorphism from $A$ into $\Compact_B(F,F)$.
  
 \item The left action of $B$ on $F^*$ is through a $*$-homomorphism from $B$ into $\Compact_A (F^*,F^*)$. 
 
\end{enumerate}
 When these conditions are met, the adjoint functor from $\H_B$ to $\H_A$ is given by tensor product with the correspondence $F^*$.\hfill\qed
\end{theorem}

Condition (b) will be explained in the next section, where we  shall also give  a proof of   Theorem~\ref{adj_bimod_theorem} in the course of our study of the weaker notion of local adjunction.  The full relationship of our results to those of \cite{KPW} will be  discussed in detail in Section \ref{sec-indexes}.

\begin{example}
\label{unital_example}
Consider the case of unital $C^*$-algebras $A$ and $B$. If $F$ is a correspondence satisfying condition (a) of Theorem \ref{adj_bimod_theorem}, then the conditions (c) and (d) are equivalent to requiring that $F$ be finitely generated (and hence projective, cf. \cite[Theorem 8.1.27]{BlM}) as a module over $B$ and $A$, respectively, while the analytic condition (b) follows automatically from (c) and (d): see \cite[Example 2.31]{KPW}. (We note that in the non-unital setting, the conditions (b), (c) and (d) are independent of one another: see the examples in Sections \ref{Hilbert_sum_example} and \ref{nonunital_noncounital_example}.) Combined with a theorem of Morita on adjunctions in the algebraic setting \cite[Theorem 4.1]{Morita}, we find that the following are equivalent for unital $C^*$-algebras $A$ and $B$:
 \begin{enumerate}[\rm (a)]
  \item The tensor product functor $\otimes_A F:\H_A\to \H_B$ has an adjoint.
  \item The algebraic tensor product functor $\otimes_A^{\alg} F:\Mod_A\to \Mod_B$ has a two-sided adjoint.
\end{enumerate}
This situation is reminiscent of Beer's result, that two unital $C^*$-algebras are Morita equivalent as $C^*$-algebras if and only they are Morita equivalent as rings \cite[Theorem 1.8]{Beer}. 
\end{example}

\subsection{Hilbert Modules as Operator Spaces}\label{operator_section}

Recall that an \emph{operator space structure} on a vector space $X$ is a sequence of Banach space norms on the spaces $M_n(X)$ of $n\times n$ matrices with entries from $X$ such that 
\begin{enumerate}[\rm (a)]
\item the norm of a block diagonal matrix is the largest of the norms of its diagonal blocks; and 
\item the norm of a three-fold product $ABC$, where $A$ and $C$ are scalar $n\times n$ matrices, and $B$ is an $n\times n$  matrix with entries from $X$, is no more than the product of the  norms of the matrices $A$, $B$ and $C$ (we use the operator norm for the scalar matrices).
\end{enumerate}
See \cite{MR1793753}  or \cite{BlM}.  An \emph{operator space} is of course a vector space with an operator space structure.  

\begin{example}\label{C*_operator_example}
The above conditions hold when $X$ is a closed subspace of a $C^*$-algebra $B$ and the norms on $M_n (X)$ are inherited from the $C^*$-algebra norm on $M_n(B)$.  
\end{example}

A linear map $T:X\to Y$ between  operator spaces is \emph{completely bounded} if 
\[
\|T\|_{\cb}= \sup_n \|M_n(T)\|<\infty,\]
 where $M_n(T):M_n(X)\to M_n(Y)$ is defined by applying $T$ entrywise.  An \emph{isomorphism of operator spaces} will mean, for us, a linear isomorphism $T:X\to Y$ such that both $T$ and $T^{-1}$ are completely bounded.  We shall use the term \emph{complete isometric isomorphism}  when 
 \[
 \|T\|_{\cb} = \|T^{-1}\|_{\cb} =1.
 \]
\begin{example}
Every operator space is completely isometrically isomorphic to a closed subspace of a $C^*$-algebra. See for example \cite[Chapter 2]{MR1793753} for an exposition.
\end{example}

Suppose now that $X$ is a right Hilbert $A$-module.  Then $M_n(X)$ is a Hilbert $M_n(A)$-module in a natural way, namely the right action of $M_n(A)$ on $M_n(X)$ is given by matrix multiplication, and the $M_n(A)$-valued inner product is  
\[
\langle S, T\rangle _{i,j} = \sum _{k=1}^n \langle S_{kj}, T_{kj}\rangle .
\]
This observation  gives  $X$ an  {operator space} structure.  See \cite[8.2.1]{BlM}.

\begin{example}
If $A$ is a $C^*$-algebra, then the operator space norms on $M_n(A)$ induced by the inner product $\langle a,b\rangle = a^*b$ on $A$ are the canonical $C^*$-algebra norms.
\end{example}

\begin{example}\label{operator_corner_example}
If $X$ is a closed subspace of $\Bounded_A(Y,Z)$, where $Y$ and $Z$ are Hilbert $A$-modules, then the operator space structure on $X$ associated to the $\Bounded_A(Y)$-valued inner product $\langle S, T\rangle = S^*T$ coincides with the one given by the natural embedding 
of $\Bounded_A(Y,Z)$ as a corner of the $C^*$-algebra $\Bounded_A(Y\oplus Z)$. 

For instance, if $X$ is a Hilbert $A$-module then the isomorphism $X\cong \Compact_A(A,X)$ of Lemma \ref{lem-compacts} is completely isometric for the operator space structure induced on $X$ by the $A$-valued inner product, and the operator space structure on $\Compact_A(A,X)$ induced by its embedding into $\Compact_A(A\oplus X)$.
\end{example}

\begin{example}
If $H$ is a Hilbert space, then the operator space structure associated to the inner product on $H$ is the \emph{column Hilbert space} structure \cite[1.2.23]{BlM} coming from the identification $H\cong \Bounded(\C,H)$.
\end{example}

The assignment of an operator space structure to each Hilbert module, as above, gives a functorial embedding of the category of Hilbert modules and adjointable maps into the category of operator spaces and completely bounded maps:

\begin{theorem}\label{Paschke_theorem}\textup{\cite[Proposition 8.2.2]{BlM}} 
If $X$ and $Y$ are right Hilbert $A$-modules, then every bounded $A$-linear map $X\to Y$ is completely bounded. 
\end{theorem}

We are going to use  operator spaces  to treat Hilbert $A$-modules $X$ and their conjugate modules $X^*$ introduced in Definition~\ref{conjugate_definition1}  on a somewhat equal footing.  To this end, we recall the following concept:

\begin{definition}
\label{def-operator-space-conj}
If $X$ is any operator space, then its \emph{conjugate} $X^*$ is the complex conjugate vector space, equipped with the norms 
\[
\| [x_{ij} ] \|_{M_n (X^*)} = \| [ x_{ji} ]\|_{M_n(X)},
\]
which endow it with the structure of an operator space.
\end{definition}

\begin{remark} The operator space  $X^*$ is usually called the \emph{adjoint} of $X$, but once again we shall try to avoid over-using this word in this paper. However we warn the reader that the term \emph{conjugate} as it is used in \cite{KPW} refers to adjoint functors. 
\end{remark}

\begin{example}
In the case of a  Hilbert $A$-module $X$, the operator space structure on $X^*$ provided by Definition~\ref{def-operator-space-conj} is the operator space structure we would obtain by viewing $X^*$ as a \emph{left} Hilbert $A$-module (a concept that we are avoiding in this paper).  See \cite[1.2.25 and 8.2.3(2)]{BlM}. 
\end{example}

\begin{example}
\label{Hilbert_cb_example}
Let $H$ be a complex Hilbert space, equipped with its column operator space structure. The conjugate operator space $H^*$ is, as a vector space, the same as the complex conjugate Hilbert space $\overline{H}$, and the conjugate operator space structure on $H^*$ is the same as the \emph{row Hilbert space} structure \cite[1.2.23]{BlM} on $\overline{H}$. Being a Hilbert space in its own right, $\overline{H}$ also carries a column Hilbert space structure. The identity map $I:\overline{H}\to \overline{H}$, considered as a map from the row operator space to the column operator space, has $\| I\|_{\cb}^2 = \dim H$. In particular, this map is \emph{not} completely bounded if $H$ is infinite-dimensional. See \cite[p.56]{MR1793753}.
\end{example}

\begin{example}
\label{compact_ci_example}
Let $X$ be a Hilbert $A$-module, and view $X^*$ as a conjugate operator space as above. The isomorphism $X^*\cong \Compact_A(X,A)$ of Lemma \ref{lem-compacts2} is completely isometric, where $\Compact_A(X,A)$ is viewed as a subspace of $\Compact_A(X\oplus A)$ as in Example \ref{operator_corner_example}.
\end{example}

\subsection{Operator Modules}
\label{sec-operator-modules}

If a $C^*$-algebra  $A$ (or more generally a Banach algebra of operators on a Hilbert space) acts on a Banach space $X$, then there are natural induced actions 
\[
M_n(A) \times M_n (X) \longrightarrow M_n(X)
\]
that combine the given action with matrix multiplication.  When $X$ is an operator space we shall always assume that these actions are completely contractive in the sense that 
\[
 \|a\cdot x \|_{M_n(X)} \le \| a \|_{M_n(A)} \|x\|_{M_n(X)} .
 \]
 The same will go for right actions instead of left actions, and we shall use the term \emph{operator module} to describe this situation (the term \emph{h-module} is used in \cite[Section 3.1.3]{BlM}; this is a reference to the Haagerup tensor product that we shall review below).

\subsection{The Haagerup Tensor Product}\label{haagerup_section}

There are several notions of tensor product for operator spaces. Here we shall need only the \emph{Haagerup tensor product}, which is  defined as follows. Let $A$ be a $C^*$-algebra, let $X$ be a right operator $A$-module, and let $Y$ be a left operator $A$-module. The Haagerup tensor product $X\otimes_{hA} Y$  is the completion of the algebraic tensor product $X\otimes_A^{\alg} Y$ that  is characterized by the following universal property: each bilinear map 
\[
\Phi \colon X \times Y \longrightarrow Z
\]
into an operator space for which 
\begin{enumerate}[\rm (a)]
\item   $\Phi (xa,y) = \Phi (x, ay)$, and 
\item the matrix extensions 
\[
\Phi _n \colon M_n(X) \times M_n (Y) \longrightarrow M_n(Z)
\]
satisfy
\[
 \|\Phi_n (x,y) \|_{M_n(Z)} \le \| x \|_{M_n(X)} \|y\|_{M_n(Y)} .
 \]
 for all $n$
 \end{enumerate}
extends to a complete contraction from $X\otimes _{hA} Y $ to $Z$. See \cite[3.4.2]{BlM}.

 The Haagerup tensor product is associative \cite[Theorem 3.4.10]{BlM} and  functorial in both variables with respect to completely bounded  module maps \cite[Lemma 3.4.5]{BlM}, and the natural isomorphism   on algebraic tensor products  extends to a completely isometric isomorphism
\[
(  X\otimes_{hA} Y)^*\stackrel \cong\longrightarrow  Y^* \otimes_{hA} X^*. 
\]
See \cite[1.5.9]{BlM}. 

The following theorems of Blecher relate the Haagerup tensor product to the tensor product and compact operators on Hilbert modules:

\begin{theorem}\textup{\cite[Theorem 8.2.11]{BlM}}
\label{Haagerup_theorem} 
Let $X$ be a Hilbert $A$-module, and let $F$ be a correspondence from $A$ to $B$. The identity map on $X\otimes_A^{\alg} F$ extends to a completely isometric natural isomorphism
\[
X \otimes_{hA} F  \stackrel \cong\longrightarrow  X\otimes_A F 
\]
from the Haagerup tensor product to the internal Hilbert module tensor product. \qed
\end{theorem}

\begin{theorem}\textup{\cite[Corollary 8.2.15]{BlM}}
\label{tensor_compact_theorem} 
Let $X$ and $Y$   be   Hilbert $A$-modules. There is a completely  isometric isomorphism of operator spaces
\[
Y \otimes_{hA} X^*  \stackrel \cong\longrightarrow  \Compact_A(X,Y) 
\]
mapping each elementary tensor $y\otimes x^*$ to the corresponding compact operator $y\otimes x^*$  defined in \textup{(\ref{eq-rank-one2})}. \qed 
 \end{theorem}

As in Example \ref{operator_corner_example}, $\Compact_A(X,Y)$ is a closed subspace of the $C^*$-algebra $\Compact_A(X\oplus Y)$ and it  is to be viewed as  an operator space in that way. Note that the operator space $Y\otimes_{hA} X^*$ does not depend on the inner products on $X$ and $Y$, but only on the induced operator space structures. In contrast, the action of $Y\otimes_{hA} X^*$ as operators $X\to Y$ appearing in Theorem \ref{tensor_compact_theorem} does depend on the $A$-valued inner product on $X$.

\section{Local Adjunctions for Hilbert Modules}\label{loc_adj_section}

\subsection{Definitions and Basic Properties}\label{adj_def_section}

We are ready now to introduce the main   concept of the paper.

\begin{definition}\label{ladj_definition}
 Let $A$ and $B$ be $C^*$-algebras, and let $E$ and $F$ be correspondences from $B$ to $A$ and from $A$ to $B$, respectively, determining tensor product functors 
 \[
 \H_A\stackrel{\otimes _A F}\longrightarrow  \H_B \quad \text{and} \quad   \H_B\stackrel{\otimes _B E}\longrightarrow \H_A .
 \]
 A \emph{local adjunction} between these functors is a natural isomorphism 
  \[
   \Phi_{X,Y}:\Compact_B(  X\otimes _A F , Y)\stackrel \cong \longrightarrow \Compact_A(X,  Y\otimes _B E )
  \]
that is, for each $X$ and $Y$, a continuous linear map.
\end{definition}

\begin{theorem}\label{ladj_cb_theorem}
Let $\Phi$ be a local adjunction, as in Definition~\ref{ladj_definition}.  Each of the linear maps
\[
 \Phi_{X,Y}:\Compact_B( X\otimes _A F , Y) \xrightarrow{\cong}  \Compact_A(X,  Y\otimes _B E)
\]
 is an isomorphism of operator spaces.
\end{theorem}

We will see later (Corollary \ref{ladj_cb_corollary}) that the matrix norms of $\Phi_{X,Y}$ are in fact bounded independently of $X$ and $Y$.

\begin{proof}
For a Hilbert $A$-module $X$, we let $X^\infty$ denote the orthogonal direct sum of countably many copies of $X$ (see \cite[8.1.9]{BlM}).   For each $n\geq 1$ there is a natural isometric embedding 
\[
M_n(\Compact_A(X,Z))\longrightarrow  \Compact_A(X^\infty, Z^\infty)
\]
that is defined by letting each $n\times n$ matrix act by matrix multiplication on the first $n$ copies of $X$ inside $X^\infty$, and by zero on the remaining copies. One also has isometric isomorphisms $(X\otimes_A F)^\infty \cong X^\infty \otimes_A F$. The diagram
\[ 
 \xymatrix@C=40pt{
 M_n(\Compact_B(X\otimes_A F, Y)) \ar[r]^-{M_n(\Phi_{X,Y})} \ar[d] & M_n(\Compact_A(X, Y\otimes_A E)) \ar[d] \\
 \Compact_B( X^\infty \otimes_A F, Y^\infty) \ar[r]^-{\Phi_{X^\infty, Y^\infty}} & \Compact_A(X^\infty, Y^\infty\otimes_B E)
 }
\]
commutes by the naturality of $\Phi$, showing that $\|M_n(\Phi_{X,Y})\|\leq \|\Phi_{X^\infty,Y^\infty}\|$ for every $n$. 
\end{proof}

As with adjoints, there is no distinction between left  and right local adjunctions: given a local adjunction $\Phi$ as above, we may define a second local adjunction,  
\begin{equation}
\label{eq-psi-from-phi2}
\Phi^*_{X,Y} \colon 
\Compact_A(  Y\otimes _B E, X)\stackrel \cong \longrightarrow  \Compact_A( Y,X\otimes _A F  )
\end{equation}
 by means of the formula 
\[
\Phi^*_{X,Y}:T\mapsto (\Phi_{X,Y})^{-1}(T^*)^* .
\] 
 This interchanges the roles played by the correspondences $E$ and $F$ in Definition~\ref{ladj_definition}.  And as with adjoints, it is very relevant to study units and counits associated to a local adjunction.  The following definition merely repeats Definition~\ref{def-unit-counit} in the present context.

\begin{definition}
\label{def-unit-counit2}
Let $\Phi$ be a local adjunction, as in Definition~\ref{ladj_definition}.
 \begin{enumerate}[\rm (a)]
   \item A \emph{unit} for $\Phi$ is a bounded, adjointable $A$-bimodule map 
  \[
 \eta\colon  A \longrightarrow F\otimes _B E
  \]
that defines $\Phi$ by means of the commuting diagram 
\[
\xymatrix@C=40pt{
 X\otimes _A A \ar[d]_{\cong} \ar[r]^-{\mathrm{id}_X \otimes \eta}  &  X\otimes _A F \otimes _B E \ar[d]^{ T \otimes \mathrm{id}_E} \\
 X \ar[r]_-{\Phi_{X,Y}(T)}  & Y\otimes _B E .
 }
\]
 
  \item A \emph{counit} for  $\Phi$  is a bounded, adjointable $B$-bimodule map 
  \[
 \varepsilon \colon E\otimes _A F  \longrightarrow B
  \]
that defines the inverse of the isomorphism $ \Phi_{X,Y}$ by means of the commuting diagram 
\[
\xymatrix@C=40pt{
 Y\otimes _B B\ar[d]_{\cong}   &  Y\otimes _B E \otimes _A F 
 \ar[l]_-{\mathrm{id}_Y \otimes \varepsilon}
\\
Y  & X\otimes _A F  \ar[l]^-{\Phi_{X,Y}^{-1}(T)}  \ar[u]_{ T \otimes \mathrm{id}_F} .
 }
\]

  \end{enumerate}
\end{definition}

Once again, these definitions are symmetric with respect to the transposition $\Phi \leftrightarrow \Phi^*$ given in (\ref{eq-psi-from-phi2}) above. 
If $\eta$ is a unit for $\Phi$, then the adjoint operator $\eta^*$ is a counit for $\Phi^*$, while if $\varepsilon$ is a counit for $\Phi$, then $\varepsilon ^*$ is a unit for $\Phi^*$.

\begin{lemma}\label{ladj_unit_lemma}
A local adjunction admits at most one unit and at most one counit.
\end{lemma}

\begin{proof}
Let $\eta:A\to F\otimes_B E$ be a unit, and let $u_\lambda$ be an approximate unit in the $C^*$-algebra $\Compact_B(F,F)$. The map $\eta$ is the strong-operator limit of the net $(u_\lambda\otimes \id_E)\circ \eta$, and the unit property of $\eta$ identifies this net with $\Phi_{A,F}(u_\lambda)$. Thus $\eta$ is uniquely determined by $\Phi$. 
The uniqueness of counits follows by symmetry from the uniqueness of units.
\end{proof}

A local adjunction need admit neither a unit nor a counit, or it might admit  one without the other.  But in the examples of interest to us at least one will exist. It is also the case that every local adjunction admits a bounded (but not necessarily adjointable) counit $\epsilon:F\otimes_A E\to B$. See Section~\ref{sec-local-unit-counit}.

If a functor has an adjoint, then the adjoint is unique up to a canonical natural isomorphism. Local adjoints are, in general, only unique in the following weaker sense:

\begin{lemma}\label{ladj_uniqueness_lemma}
Suppose that a correspondence $F$ from $A$ to $B$ is a local adjoint to a correspondence $E$ from $B$ to $A$, and also to a second correspondence $G$ from $B$ to $A$. There is a canonical completely bounded isomorphism $E\cong G$ of $B$-$A$ operator bimodules.
\end{lemma}

\begin{proof}
The two local adjunctions give completely bounded isomorphisms
\[ E\cong \Compact_A(A,E)\cong \Compact_B(F,B)\cong \Compact_A(A,G)\cong G\]
of $B$-$A$-bimodules. 
\end{proof}

\begin{remarks} The converse of Lemma \ref{ladj_uniqueness_lemma} is also true: Theorems \ref{Haagerup_theorem} and \ref{tensor_compact_theorem} together imply that up to natural isomorphism, $\Compact_A(X, Y\otimes_B E)$ depends only on the operator bimodule structure of $E$.  In the course of proving Theorem \ref{ladj_bimod_theorem}, below, we will in fact establish a bijection between the set of local adjunctions between $F$ and $E$, and the set of $B$-$A$ operator bimodule isomorphisms $F^*\xrightarrow{\cong} E$.

On the question of uniqueness, we will later see (Proposition \ref{counit_uniqueness_proposition}) that if there exists a counit 
\[
\epsilon: E\otimes _A F \longrightarrow B,
\]
 then the canonical isomorphism $E\cong G$ of Lemma \ref{ladj_uniqueness_lemma} is adjointable, and so furnishes a natural isomorphism between the tensor product functors $E$ and $G$. In particular, if the tensor product functor $F$ has an adjoint, then it has a unique local adjoint. In the absence of a counit, a tensor product functor may admit several  local adjoints that are not isomorphic to one another as correspondences. See Section~\ref{sec-miscellany}.  
\end{remarks}

\subsection{Local Adjunctions from Adjunctions}\label{scts_section}

\begin{lemma}\label{adj_sc_lemma} 
If  $F:\H_A\to \H_B$ is a tensor product functor with an adjoint $E\colon \H_B\to \H_A$, then 
for all $X\in \H_A$ and all $Y\in \H_B$, the natural isomorphism
 \[\Phi_{X,Y}:\Bounded_B(X\otimes_A F,Y)\to \Bounded_A(X, Y \otimes_B E)\] maps $\Compact_B(X \otimes_A F,Y)$ isomorphically onto $\Compact_A(X,Y\otimes_B E)$.
\end{lemma}

\begin{proof}
The space $\Compact_B(X\otimes_A F, Y)$ is densely spanned by operators of the form 
\[
L=y\otimes (Kx\otimes f)^*,
\]
 where $y\in Y$, $x\in X$, $f\in F$, $K\in \Compact_A(X,X)$, and we are using the notation of \eqref{eq-rank-one1}. The naturality of $\Phi$ gives
\[ \Phi_{X,Y}(L) = \Phi_{X,Y}\bigl(\, (y\otimes(x\otimes f)^*)\circ (K^*\otimes 1_F)\,\bigr) = \Phi_{X,Y}\bigl (\, y\otimes(x\otimes f)^*\, \bigr)\circ K^*,\]
which is compact because $K^*$ is.  This shows that $\Phi$ maps compact operators into compact operators, and a similar argument applied to $\Phi^{-1}$ shows that the map is a bijection.
\end{proof}

\begin{corollary}\label{adj_implies_ladj_corollary}
Every   adjunction of tensor product functors
\[
 \Phi_{X,Y} \colon \Bounded_B(  X\otimes _A F ,Y)\stackrel \cong \longrightarrow  \Bounded_A(X,  Y\otimes _B E) 
 \]
restricts to a local adjunction 
\[
 \Phi_{X,Y} \colon \Compact_B(X\otimes _A F, Y)\stackrel \cong \longrightarrow \Compact_A(X, Y\otimes _B E).
\] 
\qed
\end{corollary}

\subsection{Characterization of Local Adjunctions}\label{adj_bimod_section}

\begin{theorem}\label{ladj_bimod_theorem}
 A  tensor product functor   
from $
   \H_A$ to $ \H_B
$,
 induced from a correspondence $F$,    has a locally adjoint tensor product functor if and only if both of the following conditions are met:
 \begin{enumerate}[\rm (a)]
 
 \item The conjugate  $B$-$A$-bimodule  $ {F}^*$ carries an $A$-valued inner product making  it  into a correspondence from $B$ to $A$.
 
 \item  The conjugate operator space structure  on $F^*$ is completely boundedly equivalent to the Hilbert $A$-module operator space structure on $F^*$.
 \end{enumerate}
  \end{theorem}

\begin{proof}
First suppose that $E$ is a local adjoint to $ F$.  The    isomorphism 
\[ 
\Phi_{A,B}:\Compact_B(F,B)\stackrel \cong \longrightarrow  \Compact_A(A,E)
\]
gives, using Lemma~\ref{lem-compacts}, an isomorphism of $B$-$A$-bimodules 
\[
F^* \stackrel \cong \longrightarrow  E  
\]
that is, in addition,   an operator space isomorphism. The $A$-valued inner product on $F^*$ inherited from $E$ via this isomorphism satisfies conditions (a) and (b).   

Conversely, suppose we are given a compatible  $A$-valued inner product making $F^*$ into a correspondence from $B$ to $A$.  Let us introduce a second symbol, $E$, for this correspondence, and give $E$ the operator space structure it inherits from its Hilbert $A$-module structure (in contrast, we assign to $F^*$ the operator space structure it receives as the conjugate of $F$). We have sequences of natural isomorphisms of operator spaces
\begin{align*} \Compact_B( X\otimes _A F ,Y)
& \underset{(1)}{\cong} Y\otimes_{hB}  \left( { X\otimes_A F}\right)^*  \\
& \underset{(2)}{\cong} Y\otimes_{hB} \left(  {X\otimes_{hA} F}\right )^*  \\
& \underset{(3)}{\cong} Y\otimes_{hB} {F}^* \otimes_{hA} X^* 
\end{align*}
and 
\begin{align*}
Y\otimes_{hB} E \otimes_{hA} X^* 
& \underset{(4)}{\cong} Y\otimes_B E \otimes_{hA} X^*  \\
& \underset{(5)}{\cong} \Compact_A(X,  Y\otimes _B E) ,
\end{align*}
as follows.  The isomorphisms (1) and (5) come from  Theorem \ref{tensor_compact_theorem}, while  (2) and (4) come from Theorem \ref{Haagerup_theorem}.  The isomorphism  (3) is a result of the  compatibility of the Haagerup tensor product with the conjugation operation on operator spaces.  If we assume  that the identity map from $F^*$ to $E$ is a completely bounded isomorphism of operator spaces, then of course 
\[
Y\otimes_{hB} {F}^* \otimes_{hA} X^*  \cong Y\otimes_{hB} E \otimes_{hA} X^* ,
\]
and we can link all of the displayed isomorphisms together to obtain   a local adjunction between $F$ and $E$. 
\end{proof}

Since a local adjunction between $F$ and $E$ determines canonically an operator space isomorphism $E\cong F^*$---and conversely---we will usually just write $F^*$ instead of $E$ from now on, keeping in mind that this implies the choice of a suitable $A$-valued inner product on $F^*$. 

\begin{remarks}
The proof gives a very simple formula for the adjunction isomorphism 
\[
\Phi_{X,Y}:\Compact_B( X\otimes _A F , Y)\to \Compact_A(X,  Y\otimes _B F^*) ,
\]
namely 
\begin{equation}\label{Phi_equation}
 \Phi_{X,Y}\bigl (\, y \otimes  (x\otimes f)^*\,\bigr) = (y \otimes f^* )\otimes x^* .
\end{equation}
We also note that, given a local adjunction between $F$ and $F^*$, the equivalence of the two operator space structures on $F^*$ allows us to combine Theorems \ref{Haagerup_theorem} and \ref{tensor_compact_theorem} to obtain canonical isomorphisms
\begin{equation}\label{ladj_compact_equation}
\Compact_B(F,F)\cong F\otimes_B F^* \qquad\text{and}\qquad \Compact_A(F^*,F^*)\cong F^*\otimes_A F
\end{equation}
of $A$-$A$-bimodules and $B$-$B$-bimodules, respectively. 
\end{remarks}

\begin{corollary}\label{ladj_cb_corollary}
 Let $\Phi$ be a local adjunction, as in Definition~\ref{ladj_definition}.  The matrix norms of the isomorphisms
$\Phi_{X,Y}$ are bounded independently of $X$ and $Y$:
\[ \| \Phi_{X,Y}\|_{\cb} \leq \|\Phi_{A,B}\|_{\cb} \leq \|\Phi_{A^\infty,B^\infty}\|.\]
\end{corollary}

\begin{proof}
Using Theorem \ref{tensor_compact_theorem} to identify spaces of compact operators with Haagerup tensor products, the map $\Phi_{X,Y}$ is given by
\[ 1_Y\otimes \Phi_{A,B} \otimes 1_{X^*}:Y\otimes_{hB}  F^* \otimes_{hA} X^* \to Y\otimes_{hB}  E \otimes_{hA} X^*.\]
The functoriality of the Haagerup tensor product (see \cite[Lemma 3.4.5]{BlM}) then gives $\|\Phi_{X,Y}\|_{\cb} \leq \|\Phi_{A,B}\|_{\cb}$. Theorem \ref{ladj_cb_theorem} gives the second inequality, $\|\Phi_{A,B}\|_{\cb}\leq \|\Phi_{A^\infty,B^\infty}\|$.
\end{proof}

\subsection{Local Adjunctions and  Representations}

In addition to the tensor product functors $ F:\H_A\to \H_B$ studied in the previous sections, every correspondence $F$ from $A$ to $B$ induces a functor 
\[
F = F\otimes_B : {}_B\H\longrightarrow  {}_A\H
\] 
between the categories of (nondegenerate) Hilbert space representations of $A$ and $B$.  In this section we shall prove the following result.

\begin{theorem}
\label{rep-adjunction-theorem}
Let $F$ be a correspondence from $A$ to $B$, and let $E$ be a correspondence from $B$ to $A$. Every local adjunction between $F$ and $E$ gives rise to a (two-sided) adjunction between the tensor product functors 
\[
F :{}_B\H\longrightarrow  {}_A\H \quad \text{and} \quad E :{}_A\H\longrightarrow  {}_B\H
\]
on categories of Hilbert space representations. 
\end{theorem}

\begin{remark}
Under the equivalences explained in Section \ref{sec-indexes}, this theorem corresponds to \cite[Theorem 4.4(1)]{KPW}. Compare also \cite{CH_cb}, where an adjunction theorem is formulated in the context of operator modules.
\end{remark}

Before beginning the proof let us collect some preliminary facts.

\begin{lemma}\label{ladj_ip_lemma0}
Let $F$ be a correspondence from $A$ to $B$.   The formula    
\[
( f_1^* , f_2) \mapsto  \langle f_1,f_2 \rangle \]
defines a completely contractive   map of operator $B$-$B$-bimodules 
\[
 F^*\otimes_{hA} F \longrightarrow  B .
 \] 
\end{lemma}

\begin{proof}
If $\overline{S} \in M_n (F^*)$ (we shall avoid writing $S^*$ to avoid confusion with the matrix adjoint operation) and if  $T\in M_n(F)$, then the map 
\[
\langle \argument, \argument \rangle _n \colon M_n (F^*) \times  M_n (F) \longrightarrow M_n (B)
\]
induced from the inner product map in the statement of the lemma, as in Section~\ref{haagerup_section}, sends the pair $(\overline{S},T)$ to the Hilbert module inner product 
\[
\langle S^{\top}, T\rangle \in M_n (B),
\]
where $S\in M_n(F)$ is image of $\overline S\in M_n (F^*)$ under the conjugate linear isomorphism $M_n (F^*)\to M_n (F)$ and $S^\top$ is the transpose matrix.  So the lemma follows from  the universal property of the Haagerup tensor product (cf. Section \ref{haagerup_section}) and the Cauchy-Schwarz inequality \cite[p.297]{BlM}:
\[
\|\langle S^\top, T\rangle \|_{M_n(B)} \leq \|S^\top\|_{M_n(F)} \|T\|_{M_n(F)} = \|\overline{S}\|_{M_n(F^*)} \|T\|_{M_n(F)}. \qedhere 
\]  
\end{proof}

Suppose next we are given a local adjunction 
\[
 \Phi_{X,Y}:\Compact_B( X\otimes _A F , Y) \xrightarrow{\,\,\cong\,\,}  \Compact_A(X,  Y\otimes _B E) .
\]
Use it  to identify $E$ with $F^*$ as in the proof of Theorem~\ref{ladj_bimod_theorem}, and in this way equip $F^*$ with the structure of a Hilbert $B$-$A$-bimodule.

\begin{lemma}\label{ladj_ip_lemma}
Let $F$ be a correspondence from $A$ to $B$, and assume that $F^*$ has been equipped with a Hilbert $A$-module structure making $F^*$ and $F$ local adjoints.
The formulas   
\[
\epsilon \colon f_1^*\otimes f_2 \mapsto  \langle f_1,f_2 \rangle \qquad \text{and}\qquad \delta \colon f_1\otimes f_2^* \mapsto   \langle f_1^*, f_2^*\rangle
\]
define completely bounded bimodule maps  
 \[
\epsilon :  F^*\otimes_A F \to B \qquad \text{and} \qquad \delta:F\otimes_B F^* \to A.
 \] 
\end{lemma}

\begin{proof}
The assertion about $\varepsilon$ follows from the previous lemma and Theorem~\ref{Haagerup_theorem}.  The assertion about $\delta$ follows from these results together with Theorem~\ref{ladj_bimod_theorem}.
\end{proof}

 The completely bounded map $\delta:F\otimes_B F^*\to A$ given by Lemma \ref{ladj_ip_lemma} induces, by Theorem \ref{Haagerup_theorem}, a natural bounded Hilbert space operator 
 \begin{equation}
 \label{eq-delta-transf}
 \delta_X:F\otimes_B F^* \otimes_A X\longrightarrow X
 \end{equation}
 for each $X\in {}_A\H$. Being a bounded operator between Hilbert \emph{spaces}, $\delta_X$ has an adjoint operator
 \begin{equation}
 \label{eq-eta-formula}
 \eta_X \colon X\longrightarrow F\otimes_B F^*\otimes_A X ,
  \end{equation}
  and we obtain a natural transformation  from the identity functor on ${}_A\H$ to the tensor product functor 
\[
F \circ F^*\colon X \mapsto F \otimes _B   F^* _A \otimes X.
\]

We need just such a natural transformation in order to prove that the tensor product  functor $F^*\colon {}_A \H \to {}_B \H$ is left adjoint to $F$ (namely the unit of the adjunction).  We also need a natural transformation  
 \[
  \epsilon_Y \colon  F^*\otimes_A F\otimes_B Y\longrightarrow Y
  \]
   for every $Y\in {}_B\H$ (this is the counit of the adjunction), and then we need to show that the compositions 
   \begin{equation}\label{rep_adjunction_equation}
  F\otimes_B Y \xrightarrow{\eta_{F(Y)}} F\otimes_B F^*\otimes_A F\otimes_B Y \xrightarrow{ F( \epsilon_Y)} F\otimes_B Y
\end{equation}
and 
   \begin{equation}\label{rep_adjunction_equation2}
  F^*\otimes_A X \xrightarrow{F^*(\eta_{X})} F^* \otimes _A  F\otimes_B F^*\otimes_A  X \xrightarrow{ \epsilon_{F^*(X)}} F^*\otimes_A X
\end{equation}
are the identity.   See   \cite[IV.1 Theorem 2]{MacLane}.

We shall define $\varepsilon _Y$  to be the map 
\[
    F^*\otimes_A F\otimes_B Y\xrightarrow{ \varepsilon \otimes \textrm{id}_Y} A\otimes _A Y \stackrel \cong \longrightarrow  Y
      \]
obtained from the completely bounded inner product map    $\varepsilon :F^*\otimes_A F \to B$   in Lemma~\ref{ladj_ip_lemma}.  Having done so, (\ref{rep_adjunction_equation}) and (\ref{rep_adjunction_equation2}) become effectively equivalent: just reverse the roles of $F$ and $F^*$ to get from one to the other.   So it remains to prove (\ref{rep_adjunction_equation}).
   
\begin{lemma}
\label{lem-hilb-adjunction}
Let $u_\lambda $ be an approximate unit for the $C^*$-algebra $\Compact_B (F,F)$, viewed as a net in $F\otimes _B F^*$ using the isomorphism of Theorem~\textup{\ref{tensor_compact_theorem}}.  The adjoint operator $\eta_X = \delta _X ^*$ in \textup{(\ref{eq-eta-formula})} is given by the formula 
\[
\eta_X (x) = \lim_{\lambda\to \infty} u_\lambda \otimes x,
\]
where the limit exists in the weak topology on the Hilbert space $F\otimes _B F^* \otimes _A X$.
\end{lemma}

\begin{proof}
We need to prove that 
\[
\lim _{\lambda \to \infty} \langle u_\lambda \otimes x , v\rangle = \langle x, \delta _X v \rangle
\]
for all $v\in F\otimes _B F^* \otimes _A X$.  To this end, let us show  that 
\begin{equation}
\label{eq-adjoint-T}
 \langle T \otimes x , v\rangle = \langle x, \delta _X  T^*v  \rangle ,
\end{equation}
where $T\in \Compact _B (F,F)$ and where on the right hand side of the identity the adjoint operator $T^*$ acts on the triple tensor product $F\otimes_B  F^* \otimes _A  X$ by acting on the first factor alone.   To prove (\ref{eq-adjoint-T}) it suffices to calculate with elementary tensors
\[
T = f_1  \otimes f_2 ^* \quad \text{and} \quad v = f_3 \otimes f_4 ^* \otimes x_1,
\]
and this straightforward using the formulas for  the Hilbert space inner products given in Definition~\ref{def-itp}.
\end{proof}
      
   \begin{proof}[Proof of   Theorem~\ref{rep-adjunction-theorem}]
 Let us  show that for every $Y\in {}_B \H$ the composition (\ref{rep_adjunction_equation})
is equal to the identity map. Let $f\in F$, $y\in Y$ and $z\in   F\otimes_B Y$.    It suffices to show that 
\begin{equation}
\label{eq-step-of hilb-thm0}
 \bigl\langle (F(\epsilon_Y) \eta_{F(Y)}  ( f\otimes y) , z \bigr \rangle  = \bigl\langle f\otimes y,z\bigr \rangle.
\end{equation}
If we write 
$
w = 
F(\epsilon_Y)^* z
$
then  (\ref{eq-step-of hilb-thm0}) becomes the identity 
\begin{equation}
\label{eq-step-of hilb-thm2}
 \bigl\langle  \eta_{F(Y)}  ( f\otimes y) , w \bigr \rangle  = \bigl\langle f\otimes y,z\bigr \rangle.
\end{equation}
Using the formula for $\eta_{F(Y)}$ proved in the lemma, the left hand side of (\ref{eq-step-of hilb-thm2}) is 
\begin{equation*}
\lim_{\lambda \to \infty}  \bigl\langle  u_\lambda \otimes f\otimes y , w \bigr \rangle ,
\end{equation*}
or in other words 
\begin{equation}
\label{eq-step-of hilb-thm1}
\lim_{\lambda \to \infty}  \bigl\langle F(\epsilon_Y)(  u_\lambda \otimes f\otimes y) , z \bigr \rangle .
\end{equation}
 From the definition of $\epsilon_Y$ we have
\[  F(\epsilon_Y) \colon f_1\otimes f_2^*\otimes f_3\otimes y \mapsto  f_1\langle f_2,f_3\rangle \otimes y,
\] and this implies that for every $T\in \Compact_B(F,F)\cong F\otimes_B F^*$ one has 
\begin{equation*}
F(\epsilon_Y)(T\otimes f\otimes y)=Tf\otimes y.
\end{equation*}
Applying this to (\ref{eq-step-of hilb-thm1}) we find that 
\[
 \bigl\langle (F(\epsilon_Y) \eta_{F(Y)}  ( f\otimes y) , z \bigr \rangle  
 = \lim_{\lambda \to \infty}   
   \bigl\langle   u_\lambda  f\otimes y  , z \bigr \rangle ,
\]
and (\ref{eq-step-of hilb-thm0}) follows from this. \end{proof}

The uniqueness of adjoint functors in the usual context of category theory implies:

\begin{corollary}
If $F$ admits local adjoints $E$ and $G$, then the tensor product functors 
\[E\otimes_A,\ G\otimes_A:{}_A\H\longrightarrow  {}_B\H\]
are canonically isomorphic.
\qed
\end{corollary}

\subsection{Existence of Units and Counits}
\label{sec-local-unit-counit}

We continue to work with locally adjoint Hilbert modules $F$ and $F^*$ as in Theorem \ref{ladj_bimod_theorem}, but let us return now from Hilbert space representations back to Hilbert modules.           
For each Hilbert $B$-module $Y$ we have a bounded, $B$-linear map
\[ \id_Y\otimes \epsilon: Y\otimes_B F^*\otimes_A F \to Y\otimes_B B\]
defined by identifying the Hilbert module tensor products with Haagerup tensor products, and using the functoriality of the latter with respect to completely bounded maps. A short computation using \eqref{Phi_equation} shows that for each Hilbert $A$-module $X$, and each $T\in \Compact_A(X, Y\otimes_B F^*)$, the diagram 
\[
\xymatrix@C=40pt{
 Y\otimes_B B\ar[d]_{\cong}   &  Y\otimes_B F^* \otimes_A F 
 \ar[l]_-{\id_Y \otimes \epsilon}
\\
Y  & X\otimes _B F  \ar[l]^-{\Phi_{X,Y}^{-1}(T)}  \ar[u]_{ T \otimes \id_F} .
 }
\]
is commutative: thus $\epsilon$ is almost a counit for the adjunction, its only defect being that it might not be adjointable. The action homomorphism $\eta:A\to \Bounded_B(F,F)$ may similary be considered a kind of generalised unit, a point of view justified by the following proposition.

\begin{proposition}\label{unit_proposition}
The following are equivalent:
\begin{enumerate}[\rm (a)]
  \item There exists a unit $A\to F\otimes_B F^*$ for the local adjunction.
  \item The natural isomorphism $\Phi_{X,Y}:\Compact_B(X\otimes_A F, Y)\to \Compact_A(X,Y\otimes_B F^*)$ extends to a natural transformation $\Bounded_B(X\otimes_A F, Y)\to \Bounded_A(X,Y\otimes_B F^*)$.
  \item The action of $A$ on $F$ is through a $*$-homomorphism from $A$ into $\Compact_B(F,F)$.
  \item The map $\delta:F\otimes_B F^*\to A$ of Lemma \ref{ladj_ip_lemma} is adjointable.
   \end{enumerate}
 If $A$ is unital, one may add a fifth equivalent condition:
 \begin{enumerate}[\rm (e)]
  \item $F$ is finitely generated as a right $B$-module.
 \end{enumerate}
When these conditions hold, the unit in \textup{(a)} is equal to $\delta^*$; the natural transformation in \textup{(b)} is $1$--$1$; and the $*$-homomorphism in \textup{(c)} corresponds to $\eta$ under the canonical identification of $\Compact_B(F,F)$ with $F\otimes_B F^*$.
\end{proposition}

\begin{proof}
That (a) implies (b) is clear. To prove (b) implies (c), let us first show that $\Phi_{X,Y}$ is one-to-one on $\Bounded_B(X\otimes_A F, Y)$. Fix a nonzero $T\in \Bounded_B(X\otimes_A F, Y)$, and choose $y\in Y$ such that $T^* y\neq 0$. Then the compact operator $(T^* y)^* = y^*\circ T\in \Compact_B(X\otimes_A F,B)$ is nonzero, so its image under $\Phi_{X,B}$ is nonzero. By naturality we have
\[
 \Phi_{X,B}(y^*\circ T) = (y^*\otimes 1_{F^*})\circ \Phi_{X,Y}(T),
\]
and so $\Phi_{X,Y}(T)\neq 0$. 

Now let $a$ be an element of $A= \Compact_A(A,A)$. By naturality, $\Phi_{A,F}(a\otimes 1_F)= a \circ \Phi_{A,F}(1_F)\in \Compact_A(A,F\otimes_B F^*)$. Since $\Phi_{A,F}$ is $1$--$1$, and maps $\Compact_B(F,F)$ surjectively onto $\Compact_A(A, F\otimes_B F^*)$, we conclude that $a\otimes 1_F$ is compact: i.e., $A$ acts on $F$ by compact operators.

Next, we show that (c) implies (d). Identifying $F\otimes_B F^*$ with $\Compact_B(F,F)$ as in \eqref{ladj_compact_equation}, one finds the following formula for the $A$-valued inner product:
\begin{equation}\label{compact_ip_equation}
\langle K_1, K_2 \rangle_{\Compact_B(F,F)} = \delta(K_1^*K_2).
\end{equation}
Letting $\eta:A\to \Compact_B(F,F)$ be the action homomorphism, we find that
\[ \langle \eta(a), K \rangle_{\Compact_B(F,F)} = \delta(\eta(a^*)K) = a^*\delta(K)= \langle a, \delta(K)\rangle_A,\]
where we have used that $\delta$ is a map of $A$-bimodules. Thus $\delta$ is adjointable, with adjoint $\eta$.

Finally (d) implies (a), as follows.  If $\delta$ is adjointable, then Lemma \ref{ladj_ip_lemma} implies that $\delta$ is a counit for the local adjunction $\Phi^*$, and so $\delta^*$ is a unit for $\Phi$.

If $A$ is unital, then the equivalence of (c) and (e) follows from the well-known fact (\cite[8.1.27]{BlM}) that the identity operator on a Hilbert $B$-module $F$ is compact if and only if $F$ is finitely generated.
\end{proof}

\begin{corollary}
\label{decomposition_corollary}
Let $F$ be a correspondence from $A$ to $B$, admitting a local adjoint $F^*$ and a unit $\eta:A\to F\otimes_B F^*$. The $C^*$-algebra $A$ decomposes as the direct sum of two-sided ideals
\[ A= A_{F^*}\oplus A_{F^*}^\perp\]
where $A_{F^*}= \overline{\lspan} \{ \langle f_1^*, f_2^*\rangle\ |\ f_1^*,f_2^*\in F^*\}$ and $A_{F^*}^{\perp} = \ker (A\to \Bounded_B(F,F))$.
\end{corollary}

\begin{proof}
By Proposition \ref{unit_proposition}, the existence of a unit $\eta$ implies that $\delta$ is adjointable, with $\delta^*=\eta$. Being a homomorphism of $C^*$-algebras, $\eta$ has closed range and so \cite[Theorem 3.2]{Lance} gives
\[
 A = \image \delta \oplus \ker \eta = A_{F^*}\oplus A_{F^*}^\perp
\]
as (right) Hilbert $A$-modules. Both $\delta$ and $\epsilon$ are left $A$-linear, so the above is a decomposition of $C^*$-algebras.
\end{proof}

Recalling the relationship between units and counits, Proposition \ref{unit_proposition} and Corollary \ref{decomposition_corollary} immediately give:

\begin{proposition}\label{counit_proposition}
The following are equivalent:
\begin{enumerate}[\rm (a)]
  \item There exists an \textup{(}adjointable\textup{)} counit $F^*\otimes_A F\to B$.
  \item The natural isomorphism $\Phi_{X,Y}^{-1}:\Compact_A(X, Y\otimes_B F^*)\to \Compact_B(X \otimes_A F, Y)$ extends to a natural transformation $\Bounded_A(X,Y \otimes_B F^*) \to \Bounded_B(X \otimes_A F, Y)$.
  \item The action of $B$ on $F^*$ is through a $*$-homomorphism from $B$ into the $C^*$-algebra $\Compact_A(F^*,F^*)$.
  \item The map $\epsilon:F^*\otimes_A F\to B$ of Lemma \ref{ladj_ip_lemma} is adjointable.
  \end{enumerate}
  If $B$ is unital, one may add a fifth equivalent condition:
 \begin{enumerate}[\rm (e)]
  \item $F$ is finitely generated as a left $A$-module.
 \end{enumerate}

 When these conditions hold, the counit in \textup{(a)} is equal to $\epsilon$; natural transformation in \textup{(b)} is $1$--$1$; the $*$-homomorphism in \textup{(c)} corresponds to $\epsilon^*$ under the canonical identification of $\Compact_A(F^*,F^*)$ with $F^*\otimes_A F$; and the $C^*$-algebra $B$ decomposes as the direct sum of two-sided ideals $B = B_F \oplus B_F^{\perp}$.
 \qed 
 \end{proposition}

\begin{corollary}\label{adj_iff_ladj_corollary}
A local adjunction extends to an adjunction if and only if it admits both a unit and a counit.  
\end{corollary}

\begin{proof} 
As we observed above (Proposition \ref{adj_unit_counit_proposition}), the ``only if'' direction is a standard fact about adjoint functors. For the converse, assume that a local adjunction
\[ \Phi_{X,Y}:\Compact_B(X\otimes_A F, Y) \to \Compact_A(X, Y\otimes_B F^*)\]
admits a unit and a counit. Proposition \ref{unit_proposition} gives an extension of $\Phi$ to a $1$--$1$ natural transformation $\Phi_{X,Y}:\Bounded_B(X\otimes_A F,Y)\to \Bounded_A(X, Y\otimes_B F^*)$, while Proposition \ref{counit_proposition} gives a $1$--$1$ natural transformation $\Psi_{X,Y}:\Bounded_A(X,Y\otimes F^*)\to \Bounded_B(X\otimes_A F,Y)$ such that $\Phi$ and $\Psi$ are mutually inverse on the subspaces of compact operators. For each $T\in \Bounded_B(X\otimes_A F,Y)$ and $y\in Y$ we have 
\[ y^*\circ \Psi_{X,Y}\Phi_{X,Y}(T) = \Psi_{X,B}\Phi_{X,B}(y^*\circ T) = y^* \circ T\]
in $\Compact_B(X\otimes_A F, B)$. Since the maps $y^*:Y\to B$ separate the points of $Y$, the above calculation implies that $\Psi_{X,Y}\circ \Phi_{X,Y}=\id$. A similar computation shows that $\Phi_{X,Y}\circ \Psi_{X,Y}=\id$, and so $\Phi$ is a natural isomorphism.
\end{proof}

Combining the above results, we recover Theorem~\ref{adj_bimod_theorem}.

\subsection{Uniqueness of Local Adjoints}

Local adjoints are unique up to completely bounded bimodule isomorphism, by Lemma \ref{ladj_uniqueness_lemma}. They are not, in general, unique as correspondences: see Section~\ref{sec-miscellany}. However, the stronger form of uniqueness does hold in the presence of a counit:

\begin{proposition}
\label{counit_uniqueness_proposition}
Let $F$ be a correspondence from $A$ to $B$ admitting a local adjoint $E$ and an \textup{(}adjointable\textup{)} counit 
\[
\epsilon: E\otimes_A F\longrightarrow B.
\]
 If $G$ is a second local adjoint to $F$, then $E$ and $G$ are   isomorphic as correspondences from $B$ to $A$.
\end{proposition}

The proposition  is an easy consequence of the following two lemmas.

\begin{lemma}
\label{lemma-rank-one-adjoint}
Let $E$ and $G$ be Hilbert $A$-modules and let $T\colon E\to G$ be a bounded, $A$-linear operator \textup{(}not necessarily adjointable\textup{)}.  If $K$ is any compact \textup{(}and hence adjointable\textup{)} operator on the Hilbert module $E$, and if $g\in G$, then the operator 
\begin{gather*}
 S_{K,g} \colon E \longrightarrow A \\
S_{K,g}\colon   e \longmapsto \langle g, TKe \rangle_G
 \end{gather*}
 is   adjointable    \textup{(}and indeed compact\textup{)}.
\end{lemma}

\begin{proof}
Since the compact operators are a norm-closed subspace of all operators, it suffices to prove the lemma in the special case where 
$K = e_1 \otimes e_2 ^*$.  In this case the operator in the lemma is the  compact operator
\[
S_{K,g} =  \langle g, Te_1\rangle_G \otimes e_2^* \colon E\longrightarrow A,
\]
and so the proof is complete.
\end{proof}

\begin{lemma}
\label{lemma-cb-adjoint}
Let $E$ and $G$ be correspondences from $B$ to $A$.  If the action of $B$ on $E$ is through compact operators,  then every bounded,  $A$-$B$-linear operator $T\colon E\to G$ is adjointable. 
\end{lemma}

\begin{proof}
We need to show that for every $g\in G$ there is some $f\in E$ such that 
\[
\langle g, Te\rangle_G = \langle f, e\rangle_E
\]
for every $e\in E$. By the factorization theorem (applied twice), the element $g$ may be written as a product $b_1g_1a_1$ with $b_1\in B$ and $g_1\in G$, and $a_1\in A$. Next,  if $e\in E$, then
\[
 \langle b_1g_1, Te\rangle_G = \langle g_1, b_1^* Te\rangle_G = \langle g_1, Tb_1^* e\rangle_G ,
\]
where in the last step we are using the $B$-linearity of $T$.  Since $b_1^*$ is acting on $E$ as a compact operator, it follows from the previous lemma that we can write 
\[
 a^*\langle g_1, Tb_1^* e\rangle_G  = \langle  Sa ,e\rangle_E
 \]
 for some operator $S\colon A \to E$ and every $a\in A$.  Hence if we set $f = Sa_1$, then 
 \[
\langle g, T e\rangle_G =
  \langle b_1g_1a_1, T e\rangle_G  
  =  a_1^*\langle g_1, Tb_1^* e\rangle_G
  =\langle f, e\rangle_E,
  \]
  as required.
\end{proof}

\begin{proof}[Proof of Proposition~\ref{counit_uniqueness_proposition}]
This is an immediate consequence of the previous lemma, together with Proposition~\ref{counit_proposition}, which tells us that $B$ acts on $E$ through compact operators, and Lemma~\ref{ladj_uniqueness_lemma}, which tells us that $E$ and $G$ are isomorphic as $B$-$A$-operator bimodules.
\end{proof}

\begin{corollary}
If $F$ admits an adjoint $E$, then $E$ is the unique local adjoint to $F$.\qed
\end{corollary}

\subsection{Left and Right Indexes}
\label{sec-indexes}

The purpose of this section is to explicate the relationship between our results and those of \cite{KPW}. Throughout this section, $F$ denotes a correspondence from $A$ to $B$, and we assume that $F^*$ has been equipped with an $A$-valued inner product making it into a correspondence from $B$ to $A$.

Following \cite[Definitions 2.8 and 2.9]{KPW}, we say that $F$ is of \emph{finite numerical index} if there are positive constants $l$ and $r$ such that for all $f_1,\ldots,f_n\in F$,
\[ \left\| \sum \langle f_i, f_i \rangle \right\|_B \leq l \left\| \sum f_i^*\otimes f_i \right\|_{\Compact_A(F^*,F^*)}\]
and
\[ \left\| \sum \langle f_i^*, f_i^*\rangle \right\|_A \leq r \left\| \sum f_i\otimes f_i^* \right\|_{\Compact_B(F,F)}.\]
The smallest $l$ and $r$ for which these inequalities hold are called, respectively, the \emph{left numerical index} and the \emph{right numerical index} of $F$.

\begin{lemma}
The correspondence $F$ is of finite numerical index if and only if $F^*$ is a local adjoint to $F$.
\end{lemma}

\begin{proof}
If $F^*$ is a local adjoint to $F$ then Lemma \ref{ladj_ip_lemma}, combined with the identifications $\Compact_B(F,F)\cong F\otimes_B F^*$ and $\Compact_A(F^*,F^*)\cong F^*\otimes_A F$, implies that the map 
\[ \epsilon:\Compact_A(F^*,F^*)\to B,\qquad f_1^*\otimes f_2 \mapsto \langle f_1,f_2\rangle\]
is bounded, and this ensures that $F$ has finite left numerical index (equal to $\|\epsilon\|$). Reversing the roles of $F$ and $F^*$ shows that $F$ also has finite right numerical index, equal to the norm of the map
\[ \delta:\Compact_B(F,F)\to A,\qquad f_1\otimes f_2^* \mapsto \langle f_1^*,f_2^*\rangle.\]

To prove the converse, denote by  $E$  the bimodule $F^*$ equipped with the operator space structure coming from the $A$-valued inner product; $F^*$ will denote the same bimodule with its conjugate operator space structure. If $F$ has finite right numerical index $r$, then for $f^*\in F^*$ one has
\[ \| f^*\|_E = \| \langle f^*, f^*\rangle \|_A^{1/2}\leq r \| f\otimes f^*\|_{\Compact_B(F,F)}^{1/2}= r \|\langle f,f\rangle\|_B^{1/2} = r \|f^*\|_{F^*},\]
showing that the identity map $E\to F^*$ has norm at most $r$. Similarly, if $F$ has finite left numerical index $l$, then the identity map $F^*\to E$ has norm at most $l$. Two successive applications of \cite[Corollary 2.10]{KPW} show that, if $F$ has left numerical index $l$ and right numerical index $r$, then so does $M_n(F)$ (viewed as a correspondence from $M_n(A)$ to $M_n(B)$). The above argument then shows that the identity map $M_n(E)\to M_n(F^*)$ has norm at most $r$, and its inverse has norm at most $l$. Therefore the identity $E\to F^*$ is completely bounded, and so $E$ is a local adjoint to $F$ by Theorem \ref{ladj_bimod_theorem}.
\end{proof}

Suppose $F$ has finite numerical index. The maps $\epsilon$ and $\delta$ may be extended to normal maps
\[ \epsilon'':\Compact_A(F^*,F^*)'' \to B''\qquad\text{and}\qquad \delta'':\Compact_B(F,F)''\to A''\]
between the enveloping von Neumann algebras. Following \cite[Definition 2.17]{KPW}, we say that $F$ has \emph{finite left index} if the image $\epsilon''(1)$ of the identity lies in the multiplier algebra of $B$. The condition of having \emph{finite right index} is defined analogously, in terms of the map $\delta''$. Say $F$ has \emph{finite index} if it has finite left index and finite right index.

\begin{lemma}
Suppose $F$ has finite numerical index, so that $F^*$ is a local adjoint to $F$. 
\begin{enumerate}[\rm (a)]
 \item $F$ has finite left index if and only if there is a counit $F^*\otimes_A F\to B$ for the local adjunction.
 \item $F$ has finite right index if and only if there is a unit $A\to F\otimes_B F^*$ for the local adjunction.
 \item $F$ has finite index if and only if the local adjunction is an adjunction.
\end{enumerate}
\end{lemma}

\begin{proof}
It is shown in \cite[Theorem 2.22]{KPW} that $F$ has finite left (respectively, right) index if and only if $B$ (resp. $A$) acts on $F^*$ (resp. $F$) by compact operators. The asserted equivalences now follow from Propositions \ref{unit_proposition} and \ref{counit_proposition} and from Corollary \ref{adj_iff_ladj_corollary}
\end{proof}

These results are summarised in the following table, which relates our terminology to that of \cite{KPW}:

\begin{center}
\begin{tabular}{| r | l|}
\hline
 {The functor $\otimes_A F$ has} &  {if and only if the correspondence $F$ has}\\
 \hline
 \emph{a local adjoint} &  \emph{finite numerical index} \\
 \emph{a local adjoint with unit} &   \emph{finite numerical index and finite right index} \\
 \emph{a local adjoint with counit} &  \emph{finite numerical index and finite left index} \\
 \emph{an adjoint} &  \emph{finite index} \\
 \hline
\end{tabular}
\end{center}

\section{Examples}\label{examples_section}

\subsection{Commutative Examples}
\label{sec-commutative}
Let $X$ be a compact Hausdorff space.  The category of Hilbert modules over the $C^*$-algebra $C(X)$  is equivalent to the category of continuous fields of Hilbert spaces over $X$ (and adjointable operators between continuous fields) via the functor that associates to a continuous field its Hilbert module of continuous sections.

A continuous map $\varphi\colon X \to Y$  between compact Hausdorff spaces determines a homomorphism of $C^*$-algebras from $C(Y)$ to $C(X)$.  Let
\[
 F:=C(X)
\]
be the associated  correspondence from $C(Y)$ to $C(X)$.  From the point of view of continuous fields the tensor product functor 
\[
 \phi^* : =  {\otimes} F :\H_{C(Y)}\to \H_{C(X)}
\]
associated to $F$ is given by pullback $\phi^*$ of continuous fields of Hilbert spaces along the map $\varphi$.

\paragraph*{Coverings.}

Let us examine the correspondence $F$ above in the case of a surjective map between compact Hausdorff spaces.

\begin{definition}
Let $\pi:X\to Y$ be a continuous surjection of compact Hausdorff spaces. 
\begin{enumerate}[\rm (a)]

\item The map $\pi$ is a \emph{covering} if it is an open map, and if  $\# \pi^{-1}[y]$ (the number of points in the pre image of $y$) is a finite and locally constant function of $y\in Y$.

\item The map $\pi$ is a \emph{branched covering} if it is open and if the number  $\# \pi^{-1}[y]$ is finite and uniformly bounded over $y\in Y$.

\end{enumerate}
\end{definition}

The following is a consequence of results of Pavlov and Troitsky \cite{PT}, combined with Theorem \ref{adj_bimod_theorem} and Proposition \ref{FK_proposition} (below):

\begin{proposition}
Let $\pi:X\to Y$ be a continuous surjection of compact Hausdorff spaces. 
\begin{enumerate}[\rm (a)]
\item The pullback functor $\pi^*$ admits an adjoint if and only if $\pi$ is a covering. 
\item The functor $\pi^*$ admits a local adjoint if and only if $\pi$ is a   branched covering.
 \qed
 \end{enumerate}
\end{proposition}

\begin{proof}
Theorem \ref{adj_bimod_theorem} implies that $\pi^*$ admits an adjoint if and only if $F=C(X)$ is finitely generated and projective over $C(Y)$, which by \cite[Theorem 1.3]{PT} is equivalent to $\pi$ being a covering. Proposition \ref{FK_proposition} (below) implies that $\pi^*$ admits a local adjoint if and only if there is a finite-index conditional expectation $C(X)\to C(Y)$, which by \cite[Theorem 1.1]{PT} is equivalent to $\pi$ being a branched covering.
\end{proof}

The $C(Y)$-valued inner product on the conjugate module  $F^* = C(X)^*$ is given by a formula of the kind
\[
\langle f _1^* , f_2^*  \rangle (y) =  \sum _{x\in \pi^{-1}[y]} \mu_y(x) f_1(x) \overline{ f_2(x)}
\]
where the weight functions $\mu_y:\pi^{-1}[y]\to [0,1]$ are determined (usually not uniquely) by the branching of the cover over the point $y$. For example, if $\pi$ is the quotient map for a finite group action, one can take $\mu_y$ to be the constant function $1/\#\pi^{-1}[y]$ (see Proposition \ref{W_uniqueness_proposition}). For a construction of suitable $\mu_y$ for a general branched covering, see \cite[Proof of Theorem 4.3]{PT}.

\paragraph*{Open Subsets.}
 Now let  $X$ be a compact Hausdorff space, and let $U\subset X$ be an open subset. 
 Let $F = C_0(U)$. This  is a correspondence from $C(X)$ to $C_0(U)$, whose associated tensor-product functor $\H_{C(X)}\to \H_{C_0(U)}$ is restriction of continuous fields of Hilbert spaces from $X$ to $U$ (even though $U$ is not compact it is still true that $\H_{C_0(U)}$ is equivalent to the category of continuous fields of Hilbert spaces over $U$). 
 
 \begin{proposition}
 The  functor of restriction to $U$ has a local adjoint, namely extension of continuous fields by zero, which is represented by $E=C_0(U)$ viewed as a correspondence from $C_0(U)$ to $C(X)$.  The local adjoint  is an adjoint if and only if $U$ is both open and closed in $X$.
 \end{proposition}
 
 \begin{proof}
   This is a special case of Proposition \ref{Morita_proposition}, below.
 \end{proof}

\paragraph*{Infinite-Dimensional Fibers.}
In the above examples, the bimodules that occur, viewed as continuous fields of Hilbert spaces, have finite fiber dimensions, namely fiber dimensions zero or one. Interesting examples with infinite fiber dimension do not occur in the purely commutative context.  For instance, if  $H$ is a Hilbert space, viewed as a correspondence from $\C$ to $\C$, then $H$ has a local adjoint if and only if  $H$ is finite-dimensional (in which case the local adjoint is an adjoint).

\subsection{Relation with Morita Equivalences}

Let $F$ be a correspondence from $A$ to $B$, and suppose that $F^*$ has been given an $A$-valued inner product making it into a correspondence from $B$ to $A$, such that
\begin{equation}
 \label{Morita_equation}
 \langle x^*, y^* \rangle z = x \langle y, z\rangle
\end{equation}
for all $x,y,z\in F$. This condition says that $F$ and $F^*$ restrict to mutually inverse (strong) Morita equivalences between the ideals $A_{F^*}\subseteq A$ and $B_F\subseteq B$: see \cite[8.1.2]{BlM}.

\begin{proposition}\label{Morita_proposition}
The correspondences $F$ and $F^*$ are local adjoints. They are adjoints if and only if the ideals $A_{F^*}$ and $B_F$ are complemented.
\end{proposition}

\begin{proof}
The equality \eqref{Morita_equation} implies that the two norms on $F^*$ are equal: see \cite[Lemma 8.1.15]{BlM}. The equality \eqref{Morita_equation} is also satisfied by the induced inner products on $M_n(F)$ and $M_n(F^*)$, and so the identity map on $F^*$ is a complete isometry between the two operator structures; thus $F$ and $F^*$ are local adjoints, by Theorem \ref{ladj_bimod_theorem}.

If $F$ and $F^*$ are adjoints, then the ideals $A_{F^*}$ and $B_F$ are complemented by Corollary \ref{decomposition_corollary} and Proposition \ref{counit_proposition}. Conversely, suppose that $A_{F^*}$ is complemented in $A$. The formula \eqref{Morita_equation} implies that the elements of $A_{F^*}$ act on $F$ by $B$-compact operators. On the other hand, \eqref{Morita_equation} implies that $F=\Compact_B(F,F) F = A_{F^*} F$, showing that the complementary ideal $A_{F^*}^\perp$ acts by zero on $F$. Therefore $A$ acts by compact operators on $F$, and Proposition \ref{unit_proposition} thus implies that the local adjunction admits a unit; a similar argument proves the existence of a counit.
\end{proof}

The local adjunctions for which the inner products satisfy \eqref{Morita_equation} are precisely those for which the adjunction isomorphisms
\[
   \Phi:\Compact_B(  X\otimes _A F , Y)\stackrel \cong \longrightarrow \Compact_A(X,  Y\otimes _B F^*)
\]
are isomorphisms of \emph{ternary rings of operators}, meaning that 
\begin{equation*}\label{tro_eq}
 \Phi(RS^*T) = \Phi(R)\Phi(S)^*\Phi(T)
\end{equation*}
for every $R,S,T\in \Compact_B(X\otimes_A F, Y)$. (This is an immediate consequence of the formula \eqref{Phi_equation} for $\Phi$; we owe this observation to an anonymous referee.) This is in turn equivalent to the condition that the isomorphisms $\Phi$ be isometries, or equivalently complete isometries, as follows from a theorem of Hamana, Kirchberg and Ruan \cite[Corollary 4.4.6]{BlM} and a naturality argument as in Theorem \ref{ladj_cb_theorem}.

\subsection{Conditional Expectations}
\label{conditional_section}

Let $A$ be a $*$-subalgebra of a $C^*$-algebra $C$, and assume that $A$ contains an approximate unit for $C$.  Recall that  a \emph{conditional expectation} from $C$ to $A$ is a  positive, idempotent, $A$-$A$-bimodule  map $\phi:C \to A$. 

Following Frank and Kirchberg \cite[Theorem 1]{FK}, we make the following definition.

\begin{definition}
\label{def-fk}
A conditional expectation  $\phi\colon C \to A$ is of \emph{finite index} if one of the following equivalent conditions is satisfied:
 \begin{enumerate}[\rm (a)]
  \item $\phi(c^*c)\neq 0$ whenever $c\neq 0$, and $C$ is complete in the norm 
  \[
 \|c \|_\phi ^2 =  \|\phi(c^*c)\| .
 \]
 
  \item There is a $\lambda \geq 0$ such that $\lambda\phi-\id_C\colon C\to C$ is a positive map.
  \item There is a $\kappa \geq 0$ such that $\kappa\phi-\id_C\colon C \to C $ is a completely positive map.
 \end{enumerate}
\end{definition}

\begin{proposition}
\label{FK_proposition}
Let $F$ be a correspondence from $A$ to $B$ and suppose that in fact $A\subseteq \Compact_B (F,F)$.  
The correspondence $F$ admits a local adjoint if and only if there is a finite-index conditional expectation $\phi:\Compact_B (F,F) \to A$.
\end{proposition}

\begin{proof} 
Suppose given a finite-index conditional expectation as in the statement of the proposition.    Define an $A$-valued inner product on $F^*$ by
\[
\langle f_1^*,f_2^* \rangle  =  \phi(f_1 \otimes f_2^*).
\]
 Condition (a) in Definition~\ref{def-fk}  ensures that $F^*$ is a correspondence from $B$ to $A$.  Denote by $E$ this correspondence, with the operator space structure it inherits from the $A$-valued inner product. The canonical map $F^*\to E$ is completely bounded because $\phi$, like all conditional expectations, is completely bounded. The finite index condition (c) implies that the inverse map $E\to F^*$ is also completely bounded, since for each $f^*=[f^*_{i,j}]\in M_n(F^*)$ we have the estimate
 \[
 \begin{aligned}
  \kappa\|f^*\|^2_{M_n(E)}  = \|M_n(\kappa\phi)(f\otimes_{M_n(B)} f^*)\|_{M_n(A)} &\geq \| f\otimes_{M_n(B)} f^*\|_{M_n(\Compact_B(F,F))} \\
  & =  \|f^*\|^2_{M_n(F^*)}.
 \end{aligned}
\]
Thus Theorem \ref{ladj_bimod_theorem} implies that $F$ and $E$ are local adjoints.
 
Conversely, suppose that $F$ has a local adjoint, and again denote by $E$ the conjugate space $F^*$ with the operator space structure it inherits from the given $A$-valued inner product.  The $A$-valued inner product  determines a map of operator spaces 
\[
F \otimes _B F^* \longrightarrow A
\]
thanks to Lemma~\ref{ladj_ip_lemma}, and so   thanks to Theorem~\ref{tensor_compact_theorem}, a (completely) positive map
\[
\psi\colon \Compact_B (F,F) \longrightarrow A
\]
of $A$-$A$-bimodules that sends the elementary compact operator $f_1 \otimes f_2^*$ on $F$ to $\langle f_1^*, f_2^*\rangle_A$.  

Now let 
$
T = \sum _{i=1}^n x_i \otimes y_i^* \in \Compact (F).
$
We compute that 
\[
\psi (T^*T) =  \sum_{i,j=1}^n\bigl  \langle \langle x_i, x_j\rangle  y_j^*, y_i^* \bigr\rangle_A
\]
where the inside inner product is the $B$-valued inner product of $F$, and hence, after a further computation,  that $\| \psi (T^*T)\| $ is the norm of the element
 \[
 \begin{pmatrix} 
 \langle x_1, x_1\rangle & \dots & \langle x_1,x_n\rangle \\
 \vdots && \vdots \\
  \langle x_n, x_1\rangle & \dots & \langle x_n,x_n\rangle 
  \end{pmatrix}^{\frac 12} 
  \begin{pmatrix} y_1^* \\ \vdots \\ y_n ^*\end{pmatrix} \in M_n (E)
 \]
On the other hand a third computation shows that  the $M_n(F^*)$-norm of this element is $\| T^*T\|$.  Using the fact that the norms on $E$ and $F^*$ are completely equivalent we find that 
\[
\| \psi (T^*T) \| \ge \text{constant} \cdot  \|T^*T\|
\]
for some constant independent of $T$.

The restriction of the map $\psi$ to $A\subseteq \Compact_B (F,F)$ is multiplication by some positive and central element of the multiplier algebra of $A$.  The computation above shows that this element is invertible.  Adjusting $\psi$ by multiplying with the inverse of this element we obtain a finite-index conditional expectation, as required.\end{proof}


Here is an example that we shall use when analyzing parabolic induction in Section~\ref{parabolic_section}.  
Let $W$ be a finite group acting by $*$-automorphisms on a $C^*$-algebra $B$, and also acting projectively by twisted automorphisms on a Hilbert $B$-module $F$.    This means that associated to each $w\in W$ there is a $\C$-linear operator 
\[
U_w \colon F \longrightarrow F 
\]
such that 
\begin{enumerate}[\rm (i)]
\item  $ U_w(fb )  = U_w(f)  w (b)    $ for all   $f\in F $, all $b\in B$, and 
\item $\langle U_w(f_1), U_w(f_2)\rangle = w(\langle f_1,f_2\rangle)$ for all $f_1,f_2\in F $.
\end{enumerate}
and such that moreover 
\begin{enumerate}
\item[\rm (iii)]  $U_{w_1}(U_{w_2}(f) ) = U_{w_1w_2}(f) u({w_1,w_2})$
\end{enumerate}
for some unitary $u({w_1,w_2})$ in the multiplier algebra of $B$.  The formula 
\[
w(T)  = U_wTU_w^{-1} \colon F \longrightarrow F 
\]
 defines  a genuine   action of $W$ by automorphisms on the $C^*$-algebra $\Compact_B (F,F)$.

\begin{proposition}\label{W_uniqueness_proposition}
Let $W$ be a finite group acting projectively by twisted automorphisms on a Hilbert $B$-module $F$, as above, and let 
$A$ be the fixed-point algebra $\Compact_B (F,F)^W$. Then $F$, considered as a correspondence from $A$ to $B$, admits a unique local adjoint $E$ for which the canonical action of $W$ is isometric. The isomorphisms
\[ \Phi_{X,Y}:\Compact_B(X\otimes_A F, Y) \to \Compact_A(X, Y\otimes_B E)\]
for this local adjunction satisfy
\[ \| \Phi_{X,Y} \|_{\cb} = 1 \qquad \text{and} \qquad \| \Phi_{X,Y}^{-1} \|_{\cb} \leq |W|^{1/2}.\]
\end{proposition}

\begin{remark}
The \emph{canonical action} on $E$ referred to in the statement of the proposition is the one coming from the identification of $E$ with $F^*$ that the local junction implies.  And the term \emph{isometric} refers to the given Banach space structure on $E$.
\end{remark}

\begin{proof}[Proof of the Proposition]
The conditional expectation 
\[ \phi: \Compact_B (F,F) \to A,\qquad \phi(T)\coloneq \frac{1}{|W|}\sum_{w\in W} w(T)\]
has finite index, since if $T\ge 0$ then 
\[
|W|\phi(T) -T =\sum_{w\neq 1} w(T) \ge 0.
\]
 Therefore the correspondence $F$ admits a local adjoint, namely the $B$-$A$-bimodule $F^*$, with $A$-valued inner product 
\[
\langle f_1^*,f_2^*\rangle = \phi(f_1 f_2^*).
\]
 The action of $W$ on $E$ is isometric, because $\phi$ is $W$-invariant. 

Suppose $G$ is another local adjoint to $F$, and let $\psi:\Compact_B(F,F) \to A$ denote the finite-index conditional expectation  induced by the inner product on $G$ as in Proposition \ref{FK_proposition}. The action of $W$ on $G$ is isometric if and only if $\psi$ is $W$-invariant, which is to say, if and only if $\psi=\phi$. Thus $E$ is unique.

For the bounds on the norms, recall from Corollary \ref{ladj_cb_corollary} that 
\[
\|\Phi_{X,Y}\|_{\cb} \leq \|\Phi_{A,B}\|_{\cb}.
\]
The same argument shows that $\|\Phi_{X,Y}^{-1}\|_{\cb}\leq \|\Phi_{A,B}^{-1}\|_{\cb}$. The conditional expectation $\phi $ is completely bounded, with $\cb$-norm equal to $1$, and this implies that $\Phi_{A,B}$ has $\cb$-norm $1$. The map $|W|\phi-\id:B\to B$ is completely positive, and this implies that 
\[
 \|\Phi_{A,B}^{-1}\|_{\cb}\leq |W|^{1/2},
 \]
 as required.
\end{proof}

\subsection{Direct Sums}
\label{dirsum_section}

Direct sums of local adjoints are not, in general, local adjoints: for example, $\C$ is an adjoint correspondence from $\C$ to $\C$, but the countable direct sum $\C^\infty \cong \ell^2 (\mathbb N)$ does not possess a local adjoint as correspondence from $\C$ to $\C$. We may however consider $\C^\infty$ as a correspondence on the $C^*$-algebraic direct sum $\bigoplus^{\infty}\C\cong C_0(\N)$, and  viewed in this way $\C^\infty$ does have a local adjoint, as we shall see.

Suppose we are given two $C^*$-algebra direct sums 
\[
A=\bigoplus_\alpha A_\alpha \quad \text{and} \quad B=\bigoplus_\beta B_\beta .
\]
Suppose we are given a map $\beta \mapsto \iota (\beta)$ from the index set for   the decomposition of $B$ to the index set for the decomposition of $A$.  For each index $\beta$ let $F_\beta $ be a correspondence from $A_{\iota(\beta)}$ to $B_\beta$. The \emph{external direct sum} is 
\[ 
F = \bigoplus_\beta  F_\beta    = \{ (f_\beta )\in \prod F_\beta \ : \ \|f_\beta \|\to 0 \text{ as }\beta \to \infty\},
\]
with the obvious $A$-$B$-bimodule structure and $B$-valued inner product.  
 
 We shall assume throughout the rest of this section that the map $\iota$ on indices  is finite-to-one; indeed we shall assume that it is uniformly finite-to-one, in the sense that 
\begin{equation}
\label{eq-finite-pt-inverse}
\sup _\beta \# \iota^{-1}[\beta] < \infty.
\end{equation}

\begin{proposition}
\label{dirsum_proposition}
 Let $F_\beta $ be a family of correspondences from $A_{\iota(\beta)}$ to $B_\beta$, as above, and assume that condition \textup{(\ref{eq-finite-pt-inverse})} holds.  Suppose that each $F_\beta $ has a local adjoint $E_\beta$, such that the isomorphisms 
 $\phi_{\beta}: F_\beta ^*\to E _\beta $
 satisfy
 \[ \sup_\beta  \{ \|\phi_{\beta }\|_{\cb},\, \|\phi_{\beta}^{-1}\|_{\cb} \} < \infty.\]
 Then $F$ has a local adjoint.
\end{proposition}

\begin{proof}
Let  
\[
E = \bigoplus_\beta  E_\beta    = \{ (e_\beta )\in \prod E_\beta \ : \ \|e_\beta \|\to 0 \text{ as }\beta \to \infty\},
\] 
This is a correspondence from $A$ to $B$, and  the map 
\[ \phi  =  \bigoplus_\beta  \phi_{\beta} : \bigoplus_\beta  F_\beta^* \to \bigoplus E_\beta \] is an isomorphism of $B$-$A$-bimodules.  Moreover 
\[
F^* = \bigoplus _\beta F^*_\beta,
\] 
and so $\phi$ is also an isomorphism of operator spaces.
\end{proof}

\subsection{Miscellany}
\label{sec-miscellany}

\subsubsection{Units and Counits}
\label{nonunital_noncounital_example}
Let $B$ be an ideal in $A$, and consider $F=B$ as a correspondence from $A$ to $A$, with the obvious bimodule structure and the inner product $\langle b_1,b_2\rangle = b_1^*b_2$. Then $F$ is locally adjoint to itself, but the local adjunction admits neither a unit nor a counit unless $B$ is complemented in $A$. Replacing one of the two acting copies of $A$ by $B$ gives an example of a local adjunction with a unit but no counit, or vice versa.

\subsubsection{Distinct Local Adjoints}
Let $B=C[-1,1]$, and let $A=C[0,1]$, considered as a subalgebra of $B$ via the surjective map $x\mapsto |x|$, $[-1,1]\to [0,1]$. Let $F=B$, considered as a correspondence from $A$ to $B$. Define finite-index conditional expectations $\phi,\psi:B\to A$ by
\[ \phi(b)(y)\coloneq \frac{1}{2} b(y) + \frac{1}{2}b(-y) \qquad \text{and}\qquad \psi(b)(y)\coloneq \frac{2}{3} b(y) + \frac{1}{3} b(-y).\]
The associated correspondences $E=B_\phi$ and $G=B_\psi$ from $B$ to $A$ are both locally adjoint to $F$, but $E$ and $G$ are not isomorphic as Hilbert modules. Indeed, if they were isomorphic, then by polar decomposition we could find a unitary isomorphism of bimodules $U:E\to G$. Unitarity means that $\phi(b)=\psi(b U(1)^* U(1))$ for every $b\in B$, but this equality implies that the function $U(1)^* U(1)$ is discontinuous at the origin.

\subsubsection{Action by Compact Operators}
\label{Hilbert_sum_example}
Let $A_i=B_i=\C$ for $i=1,2,3,\ldots$. Let $F_i=\C^i$, with its canonical Hilbert space structure, and let $E_i=(\C^i)^*$, with its canonical Hilbert space structure. Considered as a correspondence from $\C$ to $\C$, each $F_i$ has $E_i$ as its adjoint (cf. Example \ref{unital_example}). The direct sum $F=\bigoplus F_i$ does not, however, admit a local adjoint. Indeed, the universal property of direct sums implies that the only possibility for a local adjoint would be the direct sum $E=\bigoplus E_i$, but $\bigoplus F_i^*$ and $\bigoplus E_i$ are not isomorphic as operator bimodules over $\bigoplus_i \C^i$ (cf. Example \ref{Hilbert_cb_example}). Note that $A=\bigoplus A_i$ and $B=\bigoplus B_i$ act as compact operators on $F$ and $E$ respectively, showing that condition (b) in Theorem \ref{adj_bimod_theorem} is independent of the other conditions.

\subsubsection{Simple C*-Algebras}  Finally, here is a short remark concerning simple $C^*$-algebras.

\begin{proposition}\label{ladj_simple_proposition}
Let $\Phi$ be a local  adjunction as in Definition \ref{ladj_definition}.
  \begin{enumerate}[\rm (a)]
   \item If $A$ is simple, $\Phi$ admits a unit.
   \item If $B$ is simple, $\Phi$ admits a counit.
   \item If $A$ and $B$ are simple, $\Phi$ extends to an adjunction.
  \end{enumerate}
\end{proposition}

\begin{proof}
Item (a) follows from Proposition \ref{unit_proposition} and \cite[Corollary 2.26]{KPW}. Item (b) follows by symmetry from (a).   Item (c) follows from (a) and (b) by Corollary \ref{adj_iff_ladj_corollary}.
\end{proof}

\section{Parabolic induction}\label{parabolic_section}

In this   section we shall analyze the correspondence associated in \cite{Clare_pi} and \cite{CCH_ups} to parabolic induction of tempered unitary group representations.

We shall work with a fixed real reductive group $G$ and   standard parabolic subgroup  $P$ with Levi factor $L$ and unipotent radical $N$ (thus $P=LN $ is a semidirect product, with  $L$ acting on $N$).  
We shall be following the precise conventions of \cite[Section 3]{CCH_ups}, but for instance we might take $G$ to be $GL(n,\R)$, and   $P$  to be  a block upper triangular subgroup of $GL(n,\R)$, decomposed as a semidirect product of the block upper diagonal matrices acting on the unipotent block upper triangular matrices (with identity matrix blocks down the diagonal).

The groups $G$ and $L$ act on the homogenous space $G/N$ on the left and right, respectively. The vector space $C_c^\infty (G/N)$ carries  left and right   actions of $G$ and $L$, respectively\footnote{In the case of the $L$-action, there is an adjustment to the standard action that reflects the fact that we should identify $C_c(G/N)$ with the space of compactly supported half-densities on $G/N$.} and it may be completed to a Hilbert $C^*_r (L)$-module $C^*_r (G/N)$  on which $C^*_r (G)$ acts on the left through bounded adjointable operators.  See \cite[Section 2]{Clare_pi} or \cite[Section 4]{CCH_ups}.   It is this correspondence, from $C^*_r(G)$ to $C^*_r(L)$, that we shall study in this section. 

The importance of the correspondence  $C^*_r (G/N)$ is that   the associated  tensor product functor from Hilbert space representations of $C^*_r (L)$ to $C^*_r (G)$ may be identified with the functor of parabolic induction from tempered unitary representations of $L$ to tempered unitary representations of $G$. See \cite[Section 3]{Clare_pi}. In \cite[Definition 8.2]{CCH_ups} we defined a correspondence $C^*_r(N\backslash G)$ from $C^*_r(L)$ to $C^*_r(G)$ (the construction is outlined below). Here we shall prove the following result.

\begin{theorem} 
\label{main-reductive-thm}
The correspondences $C^*_r(G/N)$ and $C^*_r(N\backslash G)$ are local adjoints.
\end{theorem}

\begin{proof} The $C^*$-algebras $A=C^*_r (G)$ and $B= C^*_r (L)$ both decompose into direct sums 
\[
A = \bigoplus _\alpha A_\alpha \quad \text{and}  \quad B = \bigoplus_\beta  B_\beta .
\]
The index sets are ``associate classes'' of pairs consisting of a parabolic subgroup   in $G$ or $L$, respectively, and a discrete series representation of the compactly generated part of the parabolic.  See \cite[Proposition 5.16 and Theorem 6.8]{CCH_ups}.

There is a finite-to-one map $\iota$ from the index set for $B$ to the index set for $A$ which comes from the enlargement of standard parabolic subgroups of $L$ to standard parabolic subgroups of $G$ (for instance in the case of $GL(n,\R)$, a parabolic subgroup of $L$ consists of a block diagonal group with each block itself a block upper triangular group, and this may be naturally extended to a block upper triangular subgroup of $GL(n,\R)$).   The enlargement process preserves Levi factors. The map  $\iota$  of associate classes is finite-to-one because non-associate pairs for $B$ can become associate when enlarged to become pairs for $A$. There is a uniform bound on the size of the point inverse images of $\iota$.   

The correspondence $F = C^*_r (G/N)$ from $A$ to $B$  decomposes into an orthogonal  direct sum 
\[
F = \bigoplus_\beta  F_\beta 
\]
with the same index set as $B$.  The inner products of elements in $F_\beta$ lie within the summand $B_\beta \subseteq B$.  The action of $A$ on $F_\beta$ factors through the projection onto the summand $A_{\iota (\beta)}\subseteq A$.

The structure of $F_\beta$ is a little complicated to recall here (see \cite[Section 7]{CCH_ups}).  But what is important for us is the structure of the $C^*$-algebra $\Compact _B (F_\beta, F_\beta)$, which is given by a formula of the type 
\[
\Compact_B (F_\beta, F_\beta) \cong  \Compact_{C_0(Z_\beta)} (  \mathcal M_\beta , \mathcal M _\beta )^{W_\beta (L)} , 
\]
where, to use notation that has been streamlined from \cite{CCH_ups}, 
\begin{enumerate}[\rm (a)]
\item 
$\mathcal M_\beta $ is a continuous field of parabolically induced tempered unitary representations of $G$ over a locally compact space $Z_\beta$ (or in other words a correspondence from $C^*_r (G)$ to $C_0(Z_\beta)$).

\item $W_\beta (L)$ is finite group acting on $Z_\beta$ and also acting projectively by twisted automorphisms on $\mathcal M_\beta$.
 \end{enumerate}
The action of $A= C^*_r (G)$ on $F_\beta$ is through the inclusion of the summand $A_{\iota(\beta)}$ as a fixed point algebra
\[
A_{\iota(\beta)} = 
\Compact_{C_0(Z_\beta)} (  \mathcal M_\beta , \mathcal M _\beta )^{W_\beta (G)} ,
\]
where $W_\beta (G)$ is a finite group containing $W_\beta (L)$ as a subgroup that also acts on $Z_\beta$,  and    on $\mathcal M_\beta$ projectively, by twisted automorphisms.

It follows from Proposition~\ref{W_uniqueness_proposition} that $F_\beta$ has a local adjoint $F_\beta^*$, whose $A_{\iota(\beta)}$-valued inner product is given by averaging over $W_{\beta}(G)$.  The finite groups $W_\beta (G)$ are uniformly bounded in size, and so it follows from Proposition \ref{dirsum_proposition} that $F = \oplus_\beta F_\beta$ has a local adjoint too, namely the direct sum $\oplus_\beta F_\beta^*$. The correspondence $C^*_r(N\backslash G)$ was defined in \cite{CCH_ups} to be precisely this direct sum.
\end{proof}

 We obtain from  Theorem~\ref{rep-adjunction-theorem} the following consequence:
 
\begin{theorem}
\label{reductive-Hilbertspace-theorem}
The functor of parabolic induction, from tempered unitary Hilbert space representations of $L$ to tempered unitary representations of $G$, possesses a two-sided adjoint. \qed
\end{theorem}

Let us note for comparison that, while there is a $C^*$-correspondence $C^*(G/N)$ implementing parabolic induction for \emph{full} group $C^*$-algebras (see \cite{Clare_pi}), there is no analogue of Theorem \ref{main-reductive-thm} in this setting:

\begin{proposition}
 \label{no-full-adjoint-proposition}
 Let $G=\SL(2,\R)$, and let $P=LN$ be the parabolic subgroup of upper-triangular matrices. The $C^*$-correspondence $C^*(G/N)$, from $C^*(G)$ to $C^*(L)$, does not possess a local adjoint. 
\end{proposition}

\begin{proof}
Suppose that $C^*(G/N)$ does have a local adjoint.  Since $L$ is amenable we in fact have $C^*(G/N)=C^*_r(G/N)$, and the left action of $C^*(G)$ factors through the quotient mapping $C^*(G)\to C^*_r(G)$. Now $C^*_r(G)$ acts on $C^*_r(G/N)$ by compact operators (see \cite[Proposition 4.4]{CCH_ups}), and so Proposition \ref{unit_proposition} and Corollary \ref{decomposition_corollary} imply that the kernel $J$ of the action map $C^*(G)\to \Bounded_{C^*(L)}(C^*(G/N),C^*(G/N))$ is a direct summand in $C^*(G)$. The closed subset of the unitary dual $\widehat{G}$ corresponding to the ideal $J$ is precisely the  {principal series}, i.e. the set of irreducible constituents of parabolically induced representations of $G$. The principal series is not open in $\widehat{G}$, as there is a  {complementary series} of representations whose closure intersects the princial series (see e.g. \cite{Milicic}). Therefore $J$ is not a direct summand of $C^*(G)$, and $C^*(G/N)$ has no local adjoint.  
\end{proof}

We conclude with a remark.  The determination of the structure of  the correspondence $C^*_r (G/N)$ relies very heavily on the classification results on the tempered representations due to Harish-Chandra, Langlands and others.  It is an interesting open problem to approach the construction of a locally  adjoint correspondence  from a more  geometric starting point, without relying so heavily on representation theory.

 \bibliography{hmod_adj}{}
 \bibliographystyle{alpha}

\end{document}